\documentclass[a4paper,english,12pt]{article}

\usepackage[english]{babel}
\usepackage{theorem,epsfig,wrapfig,array}
\usepackage{varioref,float,amssymb,amsmath,float}

\newcommand{\me}{\mbox{\rm\tiny 1/2}}
\newcommand{\ha}{\mbox{$\cal H$}}

\newcommand{\D}{\protect\displaystyle}
\newcommand{\T}{\protect\textstyle}
\newcommand{\ipl}{\langle}
\newcommand{\ipr}{\rangle}

\newcommand{\nuA} {{\nu_A}}
\newcommand{\eps}{\varepsilon}

\newcommand{\vphi}{\varphi}
\newcommand{\lbd}{\lambda}

\newtheorem{theorem}{Theorem}
\newtheorem{lemma}[theorem]{Lemma}
\newtheorem{corollary}[theorem]{Corollary}
\newtheorem{defin}[theorem]{Definition}
\newtheorem{remark}[theorem]{Remark}
\newtheorem{applemma}{Lemma}[section]

\def\N{{\mathord{\rm I\mkern-3.6mu N}}}

\def\R{{\mathord{\rm I\mkern-3.6mu R}}}
\def\Z{{\mathord{\rm Z\mkern-5.6mu Z}}}

\begin{document}
\setcounter{footnote}{1}
% --------------------------------------------------------------------------- %
%                                    Title                                    %
%-----------------------------------------------------------------------------%
\title{A MANN ITERATIVE REGULARIZATION METHOD FOR ELLIPTIC CAUCHY PROBLEMS}

\author{H.W.\,Engl\, and\, A.\,Leit\~ao%
\footnote{On leave from Department of Mathematics, Federal University of 
Santa Catarina, P.O. Box 476, 88010-970 Florian\'opolis, Brazil}
\footnote{Supported by CNPq under grant 200049-94.1 and FWF project 
F--1308 within Spezialforschungsbereich 13} \\
Institut f\"ur Industriemathematik, \\ Johannes Kepler Universit\"at,
A--4040 Linz, Austria}

\date{}
\maketitle

\begin{abstract}
We investigate the Cauchy problem for linear elliptic operators with 
$C^\infty$--coefficients at a regular set $\Omega \subset \R^2$, which is a 
classical example of an ill-posed problem. The Cauchy data are given at the 
manifold $\Gamma \subset \partial\Omega$ and our goal is to reconstruct the 
trace of the $H^1(\Omega)$ solution of an elliptic equation at $\partial 
\Omega / \Gamma$.
The method proposed here composes the segmenting Mann iteration with a fixed 
point equation associated with the elliptic Cauchy problem. Our algorithm 
generalizes the iterative method developed by Maz'ya et al., who proposed a 
method based on solving successive well-posed mixed boundary value problems. 
We analyze the regularizing and convergence properties both theoretically 
and numerically.
\end{abstract}

\setcounter{footnote}{0}
\pagestyle{plain}
%
%
%
%-----------------------------------------------------------------------------
\section{Introduction} \label{sec:introd}

The solution of a linear elliptic Cauchy problem is written as the solution 
of a fixed point equation for an affine operator $T$ (see 
Section~\ref{sec:elcp_fpgl}). This fixed point equation is obtained as a 
byproduct in \cite{Le}.

In Section~\ref{ssec:mazya} we discuss the functional analytical formulation 
in \cite{Le} of an iterative method for solving elliptic Cauchy problems, 
which was originally proposed by Maz'ya et al. in \cite{KMF}.

The original formulation of the Mann iteration (see \cite{Ma}) is discussed 
in Section~\ref{ssec:mann}. We also analyze a variant introduced by 
Groetsch, called segmenting Mann iteration (see \cite{Gr}, \cite{EnSc}). 
These iterative methods approximate the solution of fixed point equations. 
The Groetsch variation is particularly interesting, when the fixed point 
equation is described by a non-expansive operator.

In Section~\ref{sec:mann_cp} we formulate the segmenting Mann iteration for 
the fixed point equation of Section~\ref{sec:elcp_fpgl}, associated 
with the solution of an elliptic Cauchy problem. A convergence proof for 
this new method is given in Section~\ref{ssec:converg}. Basically we choose 
a special topology for the space of boundary data, which allow us to prove 
uniform convexity and to verify some spectral properties of $T$. This 
properties are the cornerstone of the proof. A regularization property 
of the method is proved in Section~\ref{ssec:regul}. Then we provide 
convergence rates for this regularization method and report about its 
numerical performance in Section~\ref{sec:munerics}.
%
%
%
%-----------------------------------------------------------------------------
\section{Elliptic Cauchy problems and fixed point equations}
\label{sec:elcp_fpgl}

Let $\Omega \subset \R^2$ be an open bounded set and $\Gamma \subset \partial 
\Omega$ a given manifold. We denote by $P$ a second order elliptic operator 
defined in $\Omega$. We denote by {\em elliptic Cauchy problem} the following 
(time independent) initial value problem for the operator $P$ \\[2ex]
$ (CP) \hskip2cm
  \left\{ \begin{array}{rl}
    P u   = 0\, ,&\!\!\! \mbox{in } \Omega \\
    u     = f\, ,&\!\!\! \mbox{at } \Gamma \\
    u_\nu = g\, ,&\!\!\! \mbox{at } \Gamma
  \end{array} \right. $ \\[2ex]
where the given functions $f, g: \Gamma \to \R$ are called {\em Cauchy data}.

The problem we want to solve is to evaluate the trace of the solution of 
such an initial value problem at the part of the boundary where no data is 
prescribed, i.e. at $\partial\Omega \backslash \Gamma$. As a solution of 
the Cauchy problem (CP) we consider a $H^1(\Omega)$--distribution, which 
solves the weak formulation of the elliptic equation in $\Omega$ and also 
satisfies the Cauchy data at $\Gamma$ in the sense of the trace operator.

It is well known that elliptic Cauchy problems are not well posed in the 
sense of Hadamard.%
\footnote{For a formal definition of {\em well posed} problems, see e.g. 
\cite{Ba} or \cite{EHN}.}
A famous example given by Hadamard himself (see \cite{Had}) shows that we 
cannot have continuous dependence of the data. Also existence of solutions 
for arbitrary Cauchy data $(f,g)$ can not be assured,%
\footnote{The Cauchy data $(f,g)$ is called {\em consistent} if the 
corresponding problem (CP) has a solution. Otherwise $(f,g)$ is called 
{\em inconsistent} Cauchy data.}
as shows a simple argumentation with the Schwartz reflection principle 
(see \cite{GiTr}). However, extending the Cauchy--Kowalewsky and Holmgren 
Theorems to the $H^1$--context, it is possible to prove uniqueness of 
solutions in weak sense (see \cite{DaLi}, \cite{Is}).

Our next step is to characterize the solution of (CP) as solution of a 
fixed point equation. We define $\Gamma_1 := \Gamma$, $\Gamma_2 := \partial 
\Omega \backslash \Gamma$, such that $\Gamma_1\cap\Gamma_2 = \emptyset$ and 
$\overline{\Gamma_1\cup\Gamma_2} = \partial\Omega$. Further, we make the 
following assumption on the second order elliptic differential operator $P$
\begin{equation} \label{P-definition}
     P(u) \, := \, - \T\sum\limits_{i,j=1}^{2} D_i (a_{i,j} D_j u)\, ,
\end{equation}
where the operator coefficients $a_{i,j}$ satisfy
\begin{itemize}
\item $a_{i,j} \in L_\infty(\Omega)$;
\item the matrix \,$A(x) := (a_{i,j})_{i,j=1}^2$ \,satisfies: 
    \,$\xi^t A(x)\, \xi > \alpha \|\xi\|^2$, \ a.e. \,$x \in \Omega$, 
    \,$\forall \xi \in\R^2$, \,where $\alpha>0$ is given (independent of $x$).
\end{itemize}
Given the Cauchy data $(f,g) \in H^{\me}(\Gamma_1) \times H^{\me}_{00} 
(\Gamma_1)'$, we assume that there exists a $H^1$--solution of the problem%
\footnote{For details on the definition of the Sobolev spaces see \cite{Ad} 
or \cite{DaLi}.}
$$  P u = 0 \ \mbox{ in } \Omega  \, ,\ \ \ \ 
      u = f \ \mbox{ at } \Gamma_1\, ,\ \ \ \ 
 u_\nuA = g \ \mbox{ at } \Gamma_1 . $$
We are mainly interested in the determination of the Neumann trace 
$\bar{\vphi} := u_{\nuA|_{\Gamma_2}} \in H^{\me}_{00}(\Gamma_2)'$. Notice 
that, once $\bar{\vphi}$ is known, the solution of (CP) can be determined 
as the solution of the well posed mixed boundary value problem
$$  P u = 0           \ \mbox{ in } \Omega  \, ,\ \ \ \ 
      u = f           \ \mbox{ at } \Gamma_1\, ,\ \ \ \ 
 u_\nuA = \bar{\vphi} \ \mbox{ at } \Gamma_2 . $$
Next we define the operators $L_n: H^{\me}_{00}(\Gamma_2)' \to H^{\me} 
(\Gamma_2)$, $L_d: H^{\me}(\Gamma_2) \to H^{\me}_{00}(\Gamma_2)'$ by 
$L_n(\vphi) := w_{|_{\Gamma_2}}$, $L_d(\psi)  := v_{\nuA|_{\Gamma_2}}$, 
where $w, v \in H^1(\Omega)$ solve
$$  P w = 0     \ \mbox{ in } \Omega  \, ,\ \ \ \ 
      w = f     \ \mbox{ at } \Gamma_1\, ,\ \ \ \ 
 w_\nuA = \vphi \ \mbox{ at } \Gamma_2  $$
and
$$  P v = 0    \ \mbox{ in } \Omega  \, ,\ \ \ \ 
 v_\nuA = g    \ \mbox{ at } \Gamma_1\, ,\ \ \ \ 
      v = \psi \ \mbox{ at } \Gamma_2  $$
respectively. Now, defining the operator 
\begin{equation} \label{T_def}
T: H^{\me}_{00}(\Gamma_2)' \, \ni \, \vphi \, \longmapsto \,
   L_d(L_n(\vphi)) \, \in \, H^{\me}_{00}(\Gamma_2)'
\end{equation}
and observing that
$$ L_n(\bar{\vphi}) = u_{|_{\Gamma_2}}\, ,\ \ \ \ 
   L_d(u_{|_{\Gamma_2}}) = \bar{\vphi}\, , $$
we obtain the desired characterization $T(\bar{\vphi}) = \bar{\vphi}$.

Note that $T$ is an affine operator, since $L_n$ and $L_d$ are affine as 
well. The linear part of $T$ is denoted by $T_l \in {\cal L}(H^{\me}_{00} 
(\Gamma_2)')$. The affine term of the operator $T$ depends on the Cauchy 
data $(f,g)$ and is denoted by $z_{f,g} \in H^{\me}_{00}(\Gamma_2)'$. Than 
we conclude that the solution $\bar{\vphi}$ of (CP) is also a solution of 
the fixed point equation
\begin{equation} \label{fix_pkt_gl}
T \, \vphi \ (= \, T_l \, \vphi \, + \, z_{f,g}) \ = \ \vphi .
\end{equation}
The converse is also true, i.e. if $\bar{\vphi}$ is a solution of 
\eqref{fix_pkt_gl}, one concludes from the uniqueness of solution of (CP) 
that $\bar{\vphi}$ must be equivalent to $u_{|_{\Gamma_2}}$.
%
%
%
%-----------------------------------------------------------------------------
\section{Iterative methods} \label{sec:iter}

%---------------------------------------
\subsection{The Maz'ya iteration} \label{ssec:mazya}

In this section we discuss the functional analytical formulation in \cite{Le} 
of an iterative method originally proposed by Maz'ya et al. (see \cite{KMF}). 
Defining $\ha := H^{\me}_{00}(\Gamma_2)'$ and following the notation of 
Section~\ref{sec:elcp_fpgl}, the {\em Maz'ya iteration} can be written in 
the form of the algorithm:
\begin{itemize}
\item[1.] Choose $\vphi_1 \in \ha$;
\item[2.] For $k = 1, 2, \ldots$ do \\
      \mbox{\ \ } $\psi_k := L_n(\vphi_k)$; \\
      \mbox{\ \ } $\vphi_{k+1} := L_d(\psi_k)$;
\end{itemize}
Notice that this is equivalent to set $\vphi_{k+1} := T \vphi_k$, $k = 0, 1, 
\dots$, which is exactly the Picard successive approximation for equation 
\eqref{fix_pkt_gl}. In this particular case we have
\begin{equation}
\vphi_{k+1} \ = \ T^k(\vphi_1) \ = \ 
T_l^k(\vphi_1) \, + \, \T\sum\limits_{j=0}^{k-1} {T_l^j(z_{f,g})} .
\end{equation}
The choice of a special topology for the Hilbert space $\ha$ allows one to 
verify the {\em non-expansivity} and {\em asymptotic regularity} of the 
operator $T_l$.%
\footnote{The corresponding definitions can be found in \cite{BrPe}.}
This are the key properties used in \cite{Le} to prove the strong convergence 
of the sequence $\{ \vphi_k \}$ to the solution $\bar{\vphi}$ of problem (CP). 
The convergence of the Maz'ya iteration follows basically from

\begin{lemma} \label{lem:je}

Let $H$ be a Hilbert space and $A: H \to H$ a linear non-expansive regular 
asymptotic operator. Given $x \in H$, the sequence $\{ A^k x \}$ converges 
to the orthogonal projection of $x$ onto Ker$(I-A)$.%
\footnote{See \cite{Je} for a complete proof.}
\end{lemma}

The interpretation of the Cauchy problem's solution as a solution of a 
fixed point equation is already suggested in \cite{KMF}. In this paper, 
Maz'ya et al. based their argumentation on some monotonicity results and 
elliptic theory to prove the convergence of the iterative method.

%---------------------------------------
\subsection{The Mann iteration} \label{ssec:mann}

In this section we analyze the iterative method introduced by Mann (see 
\cite{Ma}) to determine solutions of fixed point equations. Let $X$ be a 
Banach space and $E \subset X$ a convex compact subset. Given a continuous 
operator $T: E \to E$, we know from Schauder's fixed point theorem that $T$ 
has at least one fixed point in $E$. The considered task is that of 
constructing in $E$ a sequence that converges to a fixed point of $T$.

The starting point for the development of the Mann iteration is the ordinary 
iteration process $x_{k+1} := T(x_k)$, with $x_1 \in E$ arbitrarily chosen.%
\footnote{It is obvious that this Picard iteration may fail to converge in 
this general framework.}
Let us introduce the infinite triangular matrix
$$ A \, = \, \left( \begin{array}{cccccc}
     1      & 0      & 0      & \cdots & 0      & \cdots \\
     a_{21} & a_{22} & 0      & \cdots & 0      & \cdots \\
     a_{31} & a_{32} & a_{33} & \cdots & 0      & \cdots \\
     \vdots & \vdots & \vdots & \ddots & \vdots & \vdots
   \end{array} \right) , $$
where the coefficients $a_{ij}$ satisfy
$$ \mbox{\it i)} \ a_{ij} \, \ge \, 0\, ,\ i,j = 1, 2, \ldots\ \ \ \ \ \ 
  \mbox{\it ii)} \ a_{ij} \, = \, 0\, ,\ j > i\ \ \ \ \ \ 
 \mbox{\it iii)} \ \T\sum\limits_{j=1}^i a_{ij} \, = \, 1\, , \ 
                 i = 1, 2, \ldots $$
The {\em Mann iteration} is defined by
\begin{itemize}
\item[1.] Choose $x_1 \in E$;
\item[2.] For $k = 1, 2, \ldots$ do \\
     \mbox{\ \ } $v_k := \sum_{j=1}^k \, a_{kj} \, x_j$; \\[1ex]
     \mbox{\ \ } $x_{k+1} := T(v_k)$;
\end{itemize}
This method is denoted briefly by $M(x_1, A, T)$. It can be regarded as a 
generalization of the ordinary iteration process, since this corresponds 
to the special choice $A = I$ (the identity matrix). In the early work 
\cite{Ma}, Mann proves the following result:

\begin{lemma} \label{lemma:mann}

If either of the sequences $\{ x_k \}$, $\{ v_k \}$ converges, than the 
other also converges to the same limit, and this common limit is a fixed 
point of $T$.
\end{lemma}

Further properties of the sets of limit points of $\{ x_k \}$ and $\{ v_k \}$ 
are also proved, under additional requirements on the coefficients $a_{ij}$.

The proof in \cite{Ma} can be extended to a locally convex Hausdorff vector 
space $X$ and $E \subset X$ a convex closed subset, by using the regularity 
of the matrix $A$. A sketch of the proof is given in \cite{Do}, which also 
analyzes the Mann iteration for {\em quasi non-expansive} operators.

In the article \cite{Gr}, Groetsch considers a variant of the iteration 
defined above. The Banach space $X$ is assumed to be {\em uniformly convex} 
and $E \subset X$ convex only. Additionally to \,{\it i)}, \,{\it ii)} \,and 
\,{\it iii)}, the further assumption
\begin{itemize}
\item[\it iv)] $a_{i+1,j} \, = \, (1 - a_{i+1,i+1}) \, a_{ij}$, \ $j \le i$
\end{itemize}
is made, in which case the matrix $A$ is said to be {\em segmenting}.%
\footnote{Due tue the geometrical interpretation, we adopt the notation 
used in \cite{Gr} and \cite{EnSc}. In \cite{Do}, matrices satisfying 
{\it iv)} are called {\em normal} matrices.}
The corresponding iterative method is called {\em Mann segmenting iteration}. 
As one can easily check, $v_{k+1}$ can be written in this case as the convex 
linear combination
\begin{equation} \label{eq:segment}
v_{k+1} \ = \ (1 - d_k) v_k \, + \, d_k T(v_k) ,
\end{equation}
where $d_k := a_{k+1,k+1}$, i.e. $v_{k+1}$ lies on the line segment joining 
$v_k$ and $x_{k+1} = T(v_k)$. Notice that the choice of the diagonal elements 
$d_k$ determines completely the segmenting matrix $A$. Next we enunciate the 
main theorem in \cite{Gr}:

\begin{lemma} \label{lemma:groe}

Let $T$ be a non-expansive operator with at least one fixed point in $E$. 
If $\sum_{k=1}^\infty \, d_k (1-d_k)$ diverges, then the sequence 
$\{ (I-T) v_k \}$ converges strongly to zero, for every $x_1 \in E$.
\end{lemma}

Notice that, differently from Lemma~\ref{lemma:mann}, compactness of $E$ is 
not required in Lemma~\ref{lemma:groe}. This is the reason, why the existence 
of fixed points of $T$ must be assumed in the last lemma.

In order to prove that $M(x_1,A,T)$ converges strongly to a fixed point of 
$T$ for every $x_1 \in E$, one needs stronger assumptions such as: \\[2ex]
\begin{tabular}{cl}
$\bullet$ &\!\!\!\! $E$ is closed convex; $T(E)$ is relatively compact in 
           $X$. \\
$\bullet$ &\!\!\!\! $E$ is closed convex; $I-T$ maps bounded closed subsets 
           of $E$ into closed \\ &\!\!\!\! subsets of $E$. \\
$\bullet$ &\!\!\!\! $E$ is closed convex; $T$ is {\em demicompact} in the 
           sense of \cite{BrPe}.
\end{tabular} \\[2ex]
One should note that the last condition is a particular case of the second one.

In the particular case $X$ is an Euclidian space, $E \subset X$ is a convex 
compact subset, $T: E \to E$ is a non-expansive mapping with a unique fixed 
point $x \in E$, and $A$ is a segmenting matrix such that $\sum_{k=1}^\infty 
\, d_k (1-d_k)$ diverges, then $M(x_1,A,T)$ converges to $x$, for every $x_1 
\in E$.

It is worth mentioning that Lemma~\ref{lemma:groe} gives on $M(x_1,A,T)$ a 
condition analogous to the asymptotic regularity, which was used in 
Lemma~\ref{lem:je} to prove the convergence of the Maz'ya iteration.
%
%
%
%-----------------------------------------------------------------------------
\section{An iterative method for Cauchy problems} \label{sec:mann_cp}

In this section we introduce a segmenting Mann iteration for solving 
the elliptic Cauchy problem (CP). Let $T$ be the operator defined in 
\eqref{T_def} and $A$ a segmenting matrix. We assume the Cauchy data 
$(f,g)$ are consistent and denote by $\bar{\vphi}$ the solution of the 
fixed point equation \eqref{fix_pkt_gl}. The {\em Mann--Maz'ya iteration} 
is defined by:
\begin{itemize}
\item[1.] Choose $\vphi_1 \in \ha$;
\item[2.] For $k = 1, 2, \ldots$ do \\
      \mbox{\ \ } $\phi_k := \sum_{j=1}^k \, a_{kj} \, \vphi_j$; \\[1ex]
      \mbox{\ \ } $\vphi_{k+1} := T(\phi_k)$;
\end{itemize}
Note that $\vphi_k, \phi_k \in \ha$, $k = 1, 2, \ldots$. We represent this 
iterative process by $(\vphi_1, A, T)$. Obviously it coincides with the 
Maz'ya iteration if one chooses $A = I$.

Defining the iteration errors $\eps_k := \vphi_k - \bar{\vphi}$ and 
$\gamma_k := \phi_k - \bar{\vphi}$, we obtain
$$  \eps_{k+1} \ = \ \vphi_{k+1} - \bar{\vphi}
    \ = \ T(\phi_k) - T(\bar{\vphi})
    \ = \ T_l(\phi_k - \bar{\vphi})
    \ = \ T_l(\gamma_k)\, , $$
$$ \gamma_k \ = \ \phi_k - \bar{\vphi} \T
   \ = \ \sum\limits_{j=1}^k \, a_{kj} \, \vphi_j - \bar{\vphi}
   \ = \ \sum\limits_{j=1}^k \, a_{kj} \, \vphi_j - 
         \sum\limits_{j=1}^k \, a_{kj} \, \bar{\vphi}
   \ = \ \sum\limits_{j=1}^k \, a_{kj} \, \eps_j . $$
It becomes clear, that the convergence of the iteration $(\vphi_1, A, T)$ to 
the fixed point $\bar{\vphi}$ is equivalent to the convergence (to zero) of 
the iteration $(\eps_1, A, T_l)$.

In order to analyze this method, we need to define a special topology for 
the space \ha.

\begin{defin} \label{def:norm}

Given $\vphi \in \ha$, let $W(\vphi) \in H^1(\Omega)$ be the solution of the 
mixed boundary value problem
$$  P w = 0     \ \mbox{ in } \Omega  \, ,\ \ \ \ 
      w = 0     \ \mbox{ at } \Gamma_1\, ,\ \ \ \ 
 w_\nuA = \vphi \ \mbox{ at } \Gamma_2 . $$
We define the functional\, $\| \cdot \|_* : \ha \to \R$\, by\, $\|\vphi\|_* 
:= \big(\int_\Omega |\nabla W(\vphi)|^2 dx\big)^{\me}$. 
\end{defin}

It is proved in \cite{Le} that the functional $\| \cdot \|_*$ defines a norm 
in \ha, which is equivalent to the usual Sobolev norm of this space. 
Actually, one can verify that the bilinear form
$$ \ipl \vphi , \psi \ipr_* \, := \,
   \int_\Omega \nabla W(\vphi) \ \nabla W(\psi) dx $$
defines an inner product in \ha. In the Hilbert space $(\ha; \ipl\cdot, 
\cdot\ipr_*)$ we are able to analyze the Mann--Maz'ya iteration 
$(\vphi_1, A, T)$.

%---------------------------------------
\subsection{Convergence proof} \label{ssec:converg}

We start this section discussing an auxiliary convexity result, that is 
needed for the proof of the main theorem.%
\footnote{For simplicity we denote by $\| \cdot \|$ the norm of the 
Hilbert space \ha, but meant is the norm $\| \cdot \|_*$ introduced in 
Definition~\ref{def:norm}.}

\begin{lemma} \label{lemma:conv}

Let $\vphi_k$, $\phi_k$ be the sequences generated by the iteration 
$(\vphi_1, A, T)$ and $\bar{\vphi}$ the solution of the fixed point equation 
\eqref{fix_pkt_gl}. If for some $\eps >0$, $k_0 \in \N$ the inequality 
$\| T(\phi_k) - \vphi_k \| \ge \eps$, $\forall$ $k \ge k_0$ holds, then 
there exists $c > 0$ such that
\begin{equation} \label{lemma:eq1}
\| \phi_{k+1} - \bar{\vphi} \| \, \le \,
   \big( 1 - c\, d_k\, (1-d_k) \big)\,
   \| \phi_k - \bar{\vphi} \| \, ,\ k \ge k_0  .
\end{equation}
\end{lemma}

\begin{proof}
Notice that $\| \phi_k - \bar{\vphi} \| \le \| \phi_1 - \bar{\vphi} \|$, 
for $k \ge 1$ (see \eqref{satz:eq3} below). Since the Hilbert space $\ha$ is 
uniformly convex, we obtain for every pair $x, y \in \ha$ with 
$\|x\|, \|y\| \le \| \phi_k - \bar{\vphi}\|$ and $\| x - y \| \ge \eps$ 
the inequality
\begin{eqnarray}
\| \lbd x + (1-\lbd)y \|
 & \le & 2\lbd\, \big\| \T\frac{1}{2} (x+y) \big\| + (1-2\lbd)\, \| y \|
         \nonumber \\
 & \le & 2\lbd\, (1-\tilde{c})\, \|\phi_k-\bar{\vphi}\| \, + \,
         (1-2\lbd)\, \|\phi_k-\bar{\vphi}\| \nonumber \\
 & \le & \|\phi_k-\bar{\vphi}\| \, \big[ 1 - 2 \lbd \tilde{c} (1-\lbd) \big]
         \, ,\ \lbd \in [0,1]\, , \label{lemma:eq2}
\end{eqnarray}
where the constant $0 < \tilde{c} < 1$ depends only on $\eps$ and 
$\| \phi_1 - \bar{\vphi} \|$. Further, it follows from \eqref{eq:segment}
\begin{equation} \label{lemma:eq3}
\phi_{k+1} - \bar{\vphi} \ = \
(1-d_k) (\phi_k - \bar{\vphi}) \, + \, d_k ( T(\phi_k) - T(\bar{\vphi}) ) .
\end{equation}
Now, from \eqref{lemma:eq3} and \eqref{lemma:eq2} with $x = T(\phi_k) - 
T(\bar{\vphi})$, $y = \phi_k - \bar{\vphi}$, $\lbd = d_k$, we obtain the 
inequality in \eqref{lemma:eq1} with $c = 2\tilde{c}$.
\end{proof}

Next we prove for the Mann--Maz'ya iteration, a result analogous to the one 
stated Lemma~\ref{lemma:groe}.

\begin{theorem} \label{satz:conv_MM}

Let $T$ be the operator defined in \eqref{T_def} and $A$ a segmenting 
matrix such that  $\sum_{k=1}^\infty d_k (1-d_k)$ diverges. The iteration 
$(\vphi_1, A, T)$ generates a sequence $\{ \phi_k \}$ such that 
$\{ (I-T) \phi_k \}$ converges strongly to zero, for every $\vphi_1 
\in \ha$.
\end{theorem}

\begin{proof}
From the segmenting property \eqref{eq:segment}, follows
\begin{equation} \label{satz:eq1}
\| \phi_{k+1} - \phi_k \| \, = \, d_k \| T(\phi_k) - \phi_k \|  .
\end{equation}
From the non-expansivity of $T_l$ and \eqref{satz:eq1} we obtain
\begin{eqnarray}
\| T(\phi_{k+1}) - \phi_{k+1} \|
 & \le & \| T(\phi_{k+1}) - T(\phi_k) \| +
         \| T(\phi_k) - \phi_{k+1} \| \nonumber \\
 & \le & d_k \| T(\phi_k) - \phi_k \| + \nonumber
         \| T(\phi_k) - (1 - d_k) \phi_k - d_k T(\phi_k) \| \\
 &  =  & d_k \| T(\phi_k) - \phi_k \| +
         (1 - d_k) \| T(\phi_k) - \phi_k \|  \nonumber \\
 &  =  & \| T(\phi_k) - \phi_k \|  . \label{satz:eq2}
\end{eqnarray}
Denoting by $\bar{\vphi}$ the solution of the fixed point equation 
\eqref{fix_pkt_gl}, we estimate
\begin{eqnarray} \label{satz:eq3}
\| \phi_{k+1} - \bar{\vphi} \|
 &  =  & \| (1-d_k) \phi_k + d_k T(\phi_k) - (1-d_k) \bar{\vphi} -
         d_k \bar{\vphi} \| \nonumber \\
 &  =  & \| (1 - d_k) (\phi_k - \bar{\vphi}) +
         d_k ( T(\phi_k) - T(\bar{\vphi}) ) \| \nonumber \\
 & \le & \| \phi_k - \bar{\vphi} \|
\end{eqnarray}
and
\begin{equation} \label{satz:eq4}
\| \phi_k - T(\phi_k) \| \, = \,
\| \phi_k - \bar{\vphi} + T(\bar{\vphi}) - T(\phi_k) \| \, \le \,
2 \| \phi_k - \bar{\vphi} \|  .
\end{equation}
Now, let us assume that $(I-T) \phi_k \not\to 0$ for some $\vphi_1$. Since 
the sequence $\| (I-T) \phi_k \|$ is monotone decreasing by \eqref{satz:eq2}, 
there exists $\eps > 0$ and $k_0 \in \N$ such that
\begin{equation} \label{satz:eq5}
\| T(\phi_k) - \phi_k \| \, \ge \, \eps\, ,\ k \ge k_0 .
\end{equation}
From \eqref{satz:eq4} and \eqref{satz:eq5} follows
\begin{equation} \label{satz:eq6}
\| \phi_k-\bar{\vphi} \| \ \ge \ \eps/2 \, ,\  k \ge k_0 .
\end{equation}
Now we obtain from \eqref{satz:eq5} and Lemma~\ref{lemma:conv}
\begin{eqnarray*}
\| \phi_{k+1} - \bar{\vphi} \|
 & \le &
 \| \phi_k - \bar{\vphi} \| \, - \, c\, d_k\, (1-d_k)\,
         \| \phi_k - \bar{\vphi} \| \\
 & \le &
 \| \phi_{k-1} - \bar{\vphi} \| - c\, d_{k-1}\, (1-d_{k-1})\,
         \| \phi_{k-1} - \bar{\vphi} \|  \\
 &     & - \, c\, d_k\, (1-d_k)\, \| \phi_k - \bar{\vphi} \| \\
 & \le &
  \| \phi_{k-1} - \bar{\vphi} \| \, - \, \| \phi_k - \bar{\vphi} \| \,
         c\, \big( d_{k-1} (1-d_{k-1}) + d_k (1-d_k) \big) \, ,
\end{eqnarray*}
for $k \ge k_0$. Repeating the argumentation we have
$$ \| \phi_{k+1} - \bar{\vphi} \|  \ \le \ 
   \| \phi_{k_0} - \bar{\vphi} \| -  c\, \| \phi_k - \bar{\vphi} \| \,
   \T\sum\limits_{j=k_0}^k d_j (1-d_j) . $$
This inequality together with \eqref{satz:eq6} imply
$$ \frac{\eps}{2} \ \le \ \| \phi_{k_0} - \bar{\vphi} \| -
                  c\, \frac{\eps}{2}\, \T\sum\limits_{j=k_0}^k d_j (1-d_j) $$
and we finally obtain
$$ \eps \Big[ 1 + c\, \T\sum\limits_{j=k_0}^k d_j (1-d_j) \Big]
   \ \le \ 2\, \| \phi_{k_0} - \bar{\vphi} \| \, ,\ k \ge k_0 . $$
This gives a contradiction, since $\sum_{j=1}^\infty d_j (1-d_j)$ diverges 
by hypothesis.
\end{proof}

Notice that to prove Theorem~\ref{satz:conv_MM}, it is enough to verify that 
$(I-T) \gamma_k \to 0$, for every $\eps_1 \in \ha$. This can be done with 
minor adaptations in the above proof.

\begin{corollary} \label{corol:conv_MM}

Let $T$ be the operator defined in \eqref{T_def} and $A$ a segmenting 
matrix such that $\sum_{k=1}^\infty d_k (1-d_k)$ diverges. For every $\vphi_1 
\in \ha$ the sequences $\{ \vphi_k \}$, $\{ \phi_k \}$ generated by the 
iteration $(\vphi_1, A, T)$ satisfy
$$ \lim_{k\to\infty} \phi_k \, = \, \bar{\vphi} \, = \,
   \lim_{k\to\infty} \vphi_k\, , $$
where $\bar{\vphi}$ is the unique determined solution of the fixed point 
equation \eqref{fix_pkt_gl}.
\end{corollary}

\begin{proof}
Since $(f,g)$ are consistent Cauchy data, existence and uniqueness of 
$\bar{\vphi}$ can be assured (see Section~\ref{sec:elcp_fpgl}). 
Since $(I-T) \phi_k = (I-T_l) (\phi_k - \bar{\vphi})$, it follows from 
Theorem~\ref{satz:conv_MM} that $\lim_k \, (I-T_l) (\phi_k - \bar{\vphi}) 
= 0$ or alternatively $\lim_k$ dist$(\phi_k-\bar{\vphi}, N) = 0$, where 
$N := $ Ker$(I-T_l)$. However we know from \cite[Theorem~2.3]{Le} that 
Ker$(I-T_l) = \{ 0 \}$, from what follows $\phi_k \to \bar{\vphi}$. The 
second statement follows from $\lim_k \, \vphi_k = \lim_k \, T(\phi_{k-1}) = 
T(\bar{\vphi}) = \bar{\vphi}$.
\end{proof}

%---------------------------------------
\subsection{A remark on noisy Cauchy data} \label{ssec:noise}

Before analyzing regularization properties and obtaining convergence rates 
for the Mann--Maz'ya iteration (see Sections~\ref{ssec:regul} and 
\ref{ssec:rate} below), it is necessary to make some considerations about 
the treatment of noisy Cauchy data.

Let $(f,g)$ be consistent Cauchy data and $z_{f,g} \in \ha$ the corresponding 
affine term of the operator $T$, defined in Section~\ref{sec:elcp_fpgl}. 
Notice that for every pair of Cauchy data $(\tilde{f},\tilde{g}) \in 
H^{\me}(\Gamma_1) \times H^{\me}_{00}(\Gamma_1)'$, consistent or not, we 
can analogously obtain a corresponding affine term $\tilde{z}$. In this 
section we investigate the following question: Given the measured data 
$(f_\eps,g_\eps)$ in $L^2(\Gamma_1) \times H^{\me}_{00} (\Gamma_1)'$, with
$$ \|f_\eps - f\|_{L^2} + \|g_\eps - g\|_{(H^{\me}_{00})'} \le \eps\, , $$
how can we obtain a corresponding affine term $z_\eps$, such that 
$\| z_{f,g} - z_\eps \| \le \eps$.

We claim that $z_\eps$ can be obtained under the following {\em a priori} 
assumption on the exact Cauchy data: $f \in H^{r}(\Gamma_1)$, $r \ge 1/2$. 
In order to verify this assertion, we first use a smoothing operator 
$S: L^2(\Gamma_1) \to H^{\me}(\Gamma_1)$ to generate a $\tilde{f_\eps} := 
S f_\eps \in H^{\me}(\Gamma_1)$, satisfying $\| f - \tilde{f_\eps} \|_{\me} 
\leq \eps'$. The existence of such an operator follows from

\begin{lemma} \label{lemma-fehler-glaetung}

Let $f \in H^r$, $r > s > 0$. There exists a smoothing operator 
$S: L^2 \to H^s$ and a positive function $\gamma$ with $\lim_{x\downarrow 0} 
\gamma(x) = 0$, such that for $\eps > 0$ and $f_\eps \in L^2$ with 
$\|f - f_\eps\|_{L^2} \leq \eps$, we have
$$ \|f - S f_\eps\|_s \ \le \ \gamma(\eps) . $$
\end{lemma}

\begin{proof}
See \cite[Lemma~14]{BaLe}.
\end{proof}

After smoothing the data $f_\eps \in L^2(\Gamma_1)$, we finally obtain from 
the Cauchy data $(\tilde{f_\eps}, g_\eps)$ a corresponding $z_\eps \in 
H^{\me}_{00} (\Gamma_1)'$ such that $|| z_{f,g} - z_\eps || < \eps'$ (note 
that the affine term $z_{f,g}$ depends continuously on the data $(f,g)$).

%---------------------------------------
\subsection{A Regularization property} \label{ssec:regul}

In this section we analyze a regularization property of the Mann--Maz'ya 
iteration. We start defining, for every fixed $\vphi \in \ha$, the family 
of operators
\begin{equation} \label{eq:Rk_def}
R_k^\vphi : \ha \ni \psi \, \longmapsto \,
            \vphi_k \in \ha\, ,\ k \in \N\, ,
\end{equation}
where $\vphi_k$ is $k$-th element of the sequence generated by the 
Mann--Maz'ya algorithm $( \vphi, A, T_l+\psi )$. In the next theorem we 
prove a regularization property of the family $\{ R_k^\vphi \}_k$ with 
respect to the solution $\bar{\vphi}$ of the fixed point equation 
\eqref{fix_pkt_gl}.

\begin{theorem} \label{satz:reg_err}

Let $\vphi$ be an arbitrary element of $\ha$ and $\{ R_k^\vphi \}_{k\in\N}$ 
be the family of operators defined in \eqref{eq:Rk_def}. There exists 
$\eps_0 > 0$ and functions $\tau: (0,\eps_0) \to \R^+$, $k: (0,\eps_0) 
\to \N$, such that \\[1ex]
\begin{tabular}{r@{\ }l}
{\it i)}  & $\tau(\eps) \to 0$, for $\eps \to 0$; \\[1ex]
{\it ii)} & For every pair of Cauchy data $(f_\eps,g_\eps)$ with 
            $\| z_\eps - z_{f,g} \| \le \eps$, we have
\end{tabular}
$$ \| R_{k(\eps)}^\vphi (z_\eps) - \bar{\vphi} \| \, \le \, \tau(\eps) . $$
\end{theorem}

\begin{proof}
Let the Cauchy data $(f_\eps,g_\eps)$ be given as in {\it ii)}. From 
the identity $\bar{\vphi} = R_k^{\bar{\vphi}}(z_{f,g})$, $k\in\N$ follows
\begin{eqnarray}
\| R_k^\vphi (z_\eps) - \bar{\vphi} \|
 &  =  & \| R_k^{\vphi - \bar{\vphi}} (z_\eps - z_{f,g}) \| \nonumber \\
 & \le & \| R_k^{\vphi - \bar{\vphi}} (0) \| +
         \| R_k^0 (z_\eps - z_{f,g}) \|  . \label{reg:eq1}
\end{eqnarray}
Using the facts: $\| T_l \| \le 1$, $\sum_{j=1}^k a_{kj} =1$, one obtains 
by induction
$$ \| R_{k+1}^0 (\psi) \| \ \le \ k \| \psi \|\, ,\ k \in \N . $$
Substituting in \eqref{reg:eq1} we have
\begin{equation} \label{reg:eq2}
\| R_k^\vphi (z_\eps) - \bar{\vphi} \| \ \le \ \eps\, k +
\| R_k^{\vphi - \bar{\vphi}} (0) \|  .
\end{equation}
Now we define for $\eps > 0$, the function
$$ \tau(\eps) \, := \, 2 \inf_{k\in\N} \,
   \{ \eps k + \| R_k^{\vphi - \bar{\vphi}} (0) \| \}  . $$
Note that the sequence $\{ R_k^{\vphi - \bar{\vphi}} (0) \}$ corresponds to 
the Mann--Maz'ya algorithm for the (consistent) Cauchy data $(f,g) = (0,0)$ 
with initial point $\vphi_1 = \vphi - \bar{\vphi}$. From Corollary~%
\ref{corol:conv_MM} follows\, $\lim_{k\to\infty} R_k^{\vphi - \bar{\vphi}}(0) 
= 0$, and we can conclude\, $\lim_{\eps\to 0} \tau(\eps) = 0$.

Given $\eps > 0$ we choose $k(\eps) \in \N$ with
$$ \eps\, k(\eps) + \| R_{k(\eps)}^{\vphi - \bar{\vphi}} (0) \|
   \, < \, \tau(\eps) $$
(this is possible from the definition of $\tau$). The theorem follows now 
from \eqref{reg:eq2}.
\end{proof}

As an immediate consequence of Theorem~\ref{satz:reg_err}, we obtain the 
following regularization property of the Maz'ya iteration:

\begin{corollary} \label{corol:reg_Maz}
The Maz'ya iteration regularizes the fixed point equation \eqref{fix_pkt_gl}, 
i.e. given $\vphi_1 \in H^{\me}_{00}(\Gamma_2)'$, there exists $\eps_0 > 0$ 
and functions $\tau: (0,\eps_0) \to \R^+$, $k: (0,\eps_0) \to \N$, such that 
$\lim\limits_{\eps\to 0} \tau(\eps) = 0$ and for every pair of Cauchy data 
$(f_\eps,g_\eps)$ with $\| z_\eps - z_{f,g} \| \le \eps$ the inequality \ 
$\| \vphi_{k(\eps)} - \bar{\vphi} \| \le \tau(\eps)$ holds.
\end{corollary}

%---------------------------------------
\subsection{Convergence rates} \label{ssec:rate}

Notice that in Theorem~\ref{satz:reg_err}, we obtain an estimate for the 
iteration error $\eps_k = \vphi_k - \bar{\vphi}$. However, without making 
any further assumption on $\eps_0$ (or equivalently on $\bar{\vphi}$), we 
cannot give concrete choices for $\tau(\cdot)$, $k(\cdot)$ as functions of 
$\eps$, i.e. we cannot prove convergence rates.

In this section we consider the iteration residual and use the {\em 
discrepancy principle} as stopping rule (again without any additional 
regularity assumption on $\bar{\vphi}$), in order to obtain a similar 
(but constructive) result (at least for the residuals; for rates of 
convergence for the iterates themselves, we will need additional 
conditions as explained in Section~\ref{ssec:source}).

We start by defining the iteration residual. Let $(f,g)$ be consistent 
Cauchy data and $z_\eps \in \ha$ with $\| z_\eps - z_{f,g} \| \le \eps$. 
Given $\vphi_1 \in \ha$, let us consider the sequences
$$ \vphi_{k+1}      \, = \, T_l\, \vphi_k + z_{f,g}\, ,\ \ \ \ 
   \vphi_{k+1}^\eps \, = \, T_l\, \vphi_k^\eps + z_\eps\, ,\ 
   k = 1, 2, \ldots , $$
$\vphi_1^\eps = \vphi_1$. The corresponding residuals (exact and real, i.e. 
using noisy data) are defined by
$$ r_k      \, := \, z_{f,g} - (I-T_l) \vphi_k\, ,\ \ \ \ 
   r_k^\eps \, := \, z_\eps - (I-T_l) \vphi_k^\eps . $$
Now let $\mu > 1$ be fixed. According to the discrepancy principle, we should 
stop the iteration at the step $k(\eps,z_\eps)$ when for the first time
$ \| r_{k(\eps,z_\eps)}^\eps \| \ \le \ \mu \eps$, i.e.
\begin{equation} \label{discrep}
k(\eps,z_\eps) \ := \ \min \{ k \in \N\ |\ 
   \| z_\eps - (I-T_l) \vphi_k^\eps \| \le \mu\eps \} .
\end{equation}

\begin{remark} \label{rem:monoton}
Notice that the residual sequences $\{ \| r_k^\eps \|\}$, $\{ \| r_k \|\}$ 
are non-increa\-sing. Indeed, this follows from
\begin{eqnarray*}
z_{f,g} - (I-T_l) \vphi_{k+1}
 & = &  z_{f,g} - (I-T_l)\, (T_l\, \vphi_k + z_{f,g}) \\
 & = &  T_l\, \big( z_{f,g} - (I-T_l) \vphi_k \big)\, , \\[1ex]
z_\eps - (I-T_l) \vphi_{k+1}^\eps
 & = &  T_l\, \big( z_\eps - (I-T_l)\, \vphi_k^\eps \big) .
\end{eqnarray*}
and the non-expansivity of $T_l$.
\end{remark}

Next we obtain an estimate for $k(\eps,z_\eps)$ in \eqref{discrep}. For 
simplicity we consider only the Maz'ya iteration ($A = I$), being the general 
case completely analog. (This result is comparable to the one known for the 
{\em Landweber iteration}; see e.g. \cite[Section~6.1]{EHN} or \cite{EnSc}.)

\begin{theorem} \label{satz:conv-rate}
If $\mu > 1$ is fixed, the stopping rule $k(\eps,z_\eps)$ determined by 
the discrepancy principle in \eqref{discrep} satisfies $k(\eps,z_\eps) = 
O(\eps^{-2})$.
\end{theorem}

\begin{proof}
Given a linear non-negative operator $T: \ha \to \ha$ we have
$$ \ipl T \psi , T \psi \ipr \, = \, \ipl \psi , \psi \ipr
   - \ipl (I-T)\psi , (I-T)\psi \ipr - 2\, \ipl (I-T)\psi , T \psi \ipr
   \, ,\ \psi \in \ha . $$
Using this identity for $T = T_l$, $\psi = \bar{\vphi} - \vphi_j$, we obtain
\begin{multline} \label{regE:eq1}
\| \bar{\vphi} - \vphi_{j+1} \|^2 \, = \, \| \bar{\vphi} - \vphi_j \|^2
 - \| z_{f,g} - (I-T_l) \vphi_j \|^2 \\
 - 2 \ipl (I-T_l) (\bar{\vphi} - \vphi_j), T_l (\bar{\vphi} - \vphi_j) \ipr
\end{multline}
(notice that\, $T \psi = \bar{\vphi} - \vphi_{j+1}$) and we can estimate
\begin{eqnarray}
\| \bar{\vphi} - \vphi_j \|^2 - \| \bar{\vphi} - \vphi_{j+1} \|^2
 &  =  & \| z_{f,g} - (I-T_l) \vphi_j \|^2 \nonumber \\
 &     & \, + \, 2\, \ipl T_l(I-T_l) (\bar{\vphi} - \vphi_j),
                        (\bar{\vphi} - \vphi_j) \ipr \nonumber \\[1ex]
 & \ge & \| z_{f,g} - (I-T_l) \vphi_j \|^2 . \label{regE:eq2}
\end{eqnarray}
Adding up this inequalities for $j = 1, \ldots, k$, we obtain
$$ \| \bar{\vphi} - \vphi_1 \|^2 - \| \bar{\vphi} - \vphi_{k+1} \|^2 \, \ge \,
   \T\sum\limits_{j=1}^k \| z_{f,g} - (I-T_l) \vphi_j \|^2 \, \ge \,
   k \| z_{f,g} - (I-T_l) \vphi_k \|^2 $$
(in the last inequality we used the monotonicity of the sequence $\| r_k \|$; 
see Remark~\ref{rem:monoton}) 
and we can conclude
\begin{equation} \label{regE:eq3}
\| z_{f,g} - (I-T_l) \vphi_k \|^2 \, \le \,
   k^{-1} \| \bar{\vphi} - \vphi_1 \|^2 .
\end{equation}
Now, let us consider the real residual $r_k^\eps$. We have
\begin{eqnarray*}
\| z_\eps - (I-T_l) \vphi_{k+1}^\eps \|
 &  =  & \| T_l^k (z_\eps - (I-T_l) \vphi_1) \| \\
 & \le & \| T_l^k (z_\eps - z_{f,g} \| \, + \,
         \| T_l^k (z_{f,g} - (I-T_l) \vphi_1) \| \\
 & \le & \eps \, + \, \| z_{f,g} - (I-T_l) \vphi_k \| .
\end{eqnarray*}
Substituting \eqref{regE:eq3} in the last inequality, we obtain
\begin{equation} \label{regE:eq4}
\| z_\eps - (I-T_l) \vphi_{k+1}^\eps \| \ \le \ 
   \eps \, + \, k^{-\frac{1}{2}} \, \| \bar{\vphi} - \vphi_1 \|  .
\end{equation}
Since the right hand side of \eqref{regE:eq4} is lower than $\mu\eps$ for 
$k > (\mu-1)^{-2} \| \bar{\vphi} - \vphi_1 \|^2 \eps^{-2}$, we have 
$k(\eps,z_\eps) \le c\, \eps^{-2}$, where the constant $c > 0$ depends only 
on $\mu$ and $\vphi_1$.
\end{proof}

From Theorem~\ref{satz:conv-rate} we obtain the desired convergence rates 
for the residuals in the Mann--Maz'ya iteration. (Notice that, as in 
Theorem~\ref{satz:reg_err}, we do not make any additional regularity 
assumption on the solution $\bar{\vphi}$.)

\begin{corollary}
Let $(f,g)$ be consistent Cauchy data, $\tau > 1$ and $\eps > 0$. Given the 
noisy data $(f_\eps,g_\eps)$, with $\| z_\eps - z_{f,g} \| \le \eps$, the 
stopping rule $k(\eps,z_\eps)$ determined by the discrepancy principle 
satisfies \\[1ex]
\begin{tabular}{r@{\ \ }l}
{\it i)}  & $\| z_\eps - (I-T_l) \vphi_{k(\eps,z_\eps)}^\eps \| \ \le \ 
            \mu \eps$; \\[1ex]
{\it ii)} & $k(\eps,z_\eps) \, = \, O(\eps^{-2})$.
\end{tabular}
\end{corollary}

%---------------------------------------
\subsection{Convergence rates under source conditions} \label{ssec:source}

In this section we again use the discrepancy principle as stopping rule for 
the iteration. However, differently from Section~\ref{ssec:rate}, we make 
additional regularity assumptions on the solution of the fixed point 
equation \eqref{fix_pkt_gl}. This assumptions are stated in the form of 
the so-called {\em source conditions}. This is a common way to insert in 
the estimates some {\em a priori} knowledge about the solution and the 
spectrum of the operator $T_l$. In this way, we can also obtain convergence 
rates for the approximate solutions, not only for the resulting residuals.

Source conditions for linear problems usually have the form
$$ \bar{\vphi} - \vphi_1 \ = \ f(T)\, \psi\, ,\ \ 
   \|\psi\| \le \widetilde{\psi} . $$
If $f(\lbd) = \lbd^\mu$ for some $\mu > 0$, we have the {\em H\"older-type 
source conditions}. Since our problem is exponentially ill-posed (the 
eigenvalues of $T_l$ converge exponentially to 1; see e.g. \cite{Le}) this 
type of condition is too restrictive. Much more natural in this case is to 
use {\em logarithmic-type source conditions}:
\begin{equation} \label{gl:log-source}
 f(\lbd) \ := \ 
   \begin{cases}
     \big( \ln (\exp(1) \lbd^{-1}) \big)^{-p} &\!\!\!\!\!\!,\ 
     \lbd > 0 \\
     \ \ \ \ \ \ \ \ \ \ \ 0 &\!\!\!\!\!\!,\ \lbd = 0
   \end{cases}
\end{equation}
with some parameter $p > 0$. (See Remark~\ref{rem:sc-interpr} for an 
interpretation of this source condition.)

\begin{theorem} \label{satz:source-rate}
Let $(f,g)$ be consistent Cauchy data and assume that the solution 
$\bar{\vphi}$ of the fixed point equation \eqref{fix_pkt_gl} satisfies 
the source condition
\begin{equation} \label{gl:source-cond}
\bar{\vphi} - \vphi_1 \; = \; f(I-T_l)\, \psi\, ,\, \mbox{ for some }
   \psi \in {\cal H}\, ,
\end{equation}
where $\vphi_1 \in {\cal H}$ is some initial guess and $f$ is the function 
defined in \eqref{gl:log-source} with $p \ge 1$. Let $\mu > 2$, 
$(f_\eps, g_\eps)$ some given noisy data with $\| z_\eps - z_{f,g} \| \le 
\eps$, $\eps > 0$ and $k(\eps,z_\eps)$ the stopping rule determined by the 
discrepancy principle. Then there exists a constant $C$, depending on $p$ 
and $\| \psi \|$ only, such that \\[1ex]
\begin{tabular}{r@{\ \ }l}
{\it i)}  & $\| \bar{\vphi} - \vphi_k^\eps \| \ \le \ C\, (\ln k)^{-p}$ \\[1ex]
{\it ii)} & $\| z_\eps - (I-T_l) \vphi_k^\eps \| \ \le \ 
            C\, k^{-1} (\ln k)^{-p}$
\end{tabular} \\[1ex]
for all iteration index $k$ satisfying $1 \le k \le k(\eps,z_\eps)$.
\end{theorem}
\begin{proof}
We set $e_k := \vphi_k^\eps - \bar{\vphi}$, $P := I-T_l$ 
(note that $P$ is self-adjoint positive non-expansive and compact). It 
follows from the definition of the iterative algorithm that
\begin{equation} \label{gl:id-ek-Pek}
\begin{array}{rl}
     e_k   & = \ (I-P)^{k-1} e_1 \, + \, \sum\limits_{j=0}^{k-2} \, (I-P)^j \,
                   (z_{f,g} - z_\eps)\, , \\[2ex]
     P e_k & = \ (I-P)^{k-1} P e_1 \, + \, \big( P - (I-P) \big)^{k-1}
                   (z_{f,g} - z_\eps)\, ,
\end{array}
\end{equation}
for $k \ge 2$. Now, from Lemma~\ref{lem:app2}, follow the estimates
\begin{equation} \label{gl:abs-ek}
\| e_k \|
 \ \le \ c \, \| \psi \| \, \big( \ln(k+1) \big)^{-p} + \, \eps\, (k-1)
\end{equation}
and
\begin{equation} \label{gl:abs-Pek}
\| P e_k \|
 \ \le \ c \, \| \psi \| \, (k-1)^{-1} \big( \ln(k+1) \big)^{-p} + \, \eps .
\end{equation}
Further, we obtain from the discrepancy principle
\begin{equation} \label{gl:abs-dp}
\mu\, \eps \ \le \ \| z_\eps - (I-T_l) \vphi_k^\eps \|
 \  =  \  \| P \vphi_k^\eps - P \bar{\vphi} + z_{f,g} - z_\eps \|
 \ \le \  \| P e_k \| + \eps .
\end{equation}
for $1 \le k < k(\eps,z_\eps)$. It follows from \eqref{gl:abs-dp} and 
\eqref{gl:abs-Pek}
\begin{equation} \label{gl:abs-eps}
\eps\, (\mu - 2) \ \le \ c \, \| \psi \| \, (k-1)^{-1}
                        \big( \ln(k+1) \big)^{-p}.
\end{equation}
Now, from \eqref{gl:abs-eps} and \eqref{gl:abs-ek} we obtain
\begin{equation} \label{gl:abs2-ek}
\| e_k \| \ \le \ c \, \| \psi \| \, \big( 1 + \T\frac{1}{\mu-2} \big)\, 
                \big( \ln(k+1) \big)^{-p} .
\end{equation}
Since $\tilde{c} := \sup_{k\in\N} \{ (\ln(k+1) / \ln(k))^{-p} \} < \infty$, 
assertion {\it i)} follows from \eqref{gl:abs2-ek} with $C = c \| \psi \| 
(1 + \frac{1}{\mu-2})\, \tilde{c}$. Now we prove {\it ii)}. It follows from 
\eqref{gl:abs-dp} and \eqref{gl:abs-Pek}
$$  \| z_\eps - (I-T_l) \vphi_k^\eps \| \ \le \ \| P e_k \| + \eps
    \ \le \ c\, \| \psi \| \, (k-1)^{-1} \big( \ln(k+1) \big)^{-p} + 2\eps $$
and together with \eqref{gl:abs-eps} we have
\begin{equation} \label{gl:abs-resid}
 \| z_\eps - (I-T_l) \vphi_k^\eps \| \ \le \ c\, \| \psi \| \big( 1 +
 \T\frac{2}{\mu-2} \big)\, (k-1)^{-1} \big( \ln(k+1) \big)^{-p} .
\end{equation}
Since $c^\star := \sup_{k\in\N} \{ (k+1) / k) \} < \infty$, assertion 
{\it ii)} follows from \eqref{gl:abs-resid} with $C = c \| \psi \| 
(1 + \frac{1}{\mu-2})\, c^\star\, \tilde{c}$.
\end{proof}

It is simple to check that, in the case of exact Cauchy data ($\eps = 0$), 
assertions {\it i)} and {\it ii)} of Theorem~\ref{satz:source-rate} hold 
for every $k \ge 1$. Therefore, in the case of exact data and under the 
source condition \eqref{gl:source-cond}, the iteration approximates the 
solution of the fixed point equation \eqref{fix_pkt_gl} with a rate of 
\,$O \big( \ln(k)^{-p} \big)$.

In the case of noisy data, the next theorem gives an estimate for the 
asymptotic behavior of the stopping index $k(\eps,z_\eps)$ in dependence 
of the noisy level $\eps$.

\begin{theorem}
Set\, $k_\eps := k(\eps,z_\eps)$. Under the assumptions of Theorem~%
\ref{satz:source-rate} we have \\[1ex]
\begin{tabular}{r@{\ \ }l}
{\it i)}  & $k_\eps \big( \ln(k_\eps) \big)^p \ = \ O ( \eps^{-1} )$; \\[1ex]
{\it ii)} & $\| \bar{\vphi} - \vphi_{k_\eps}^\eps \| \ \le \ 
            O \big(  (-\ln \sqrt{\eps})^{-p} \big)$.
\end{tabular}
\end{theorem}

\begin{proof}
In the sequence we use the notation of Theorem~%
\ref{satz:source-rate}. From inequality \eqref{gl:abs-eps} for 
$k = k_\eps - 1$ follows
$$ (\mu - 2) \eps \ \le \ c_1 (k_\eps - 2)^{-1} (\ln\, k_\eps)^{-p} $$
and we obtain
$$ \eps^{-1} \ \ge \ c_2\, (k_\eps - 2) \, (\ln\, k_\eps)^p
   \ \ge \ c_3\, k_\eps\, (\ln\, k_\eps)^p , $$
proving the first assertion. Now we prove {\it ii)}. From \eqref{gl:id-ek-Pek} 
follows
\begin{eqnarray}
e_{k_\eps}
 & = & (I-P)^{k_\eps-1} e_1 \, + \, \T\sum\limits_{j=0}^{k_\eps-2}\, (I-P)^j\,
       (z_{f,g} - z_\eps)\, , \nonumber \\
 & = & f(P) \psi_{k_\eps} \, + \, \T\sum\limits_{j=0}^{k_\eps-2} \, (I-P)^j \,
           (z_{f,g} - z_\eps)\, , \label{gl:ekeps-ident}
\end{eqnarray}
where $\psi_{k_\eps} := (I-P)^{k_\eps-1}\psi $. Then we can estimate
\begin{equation} \label{gl:ekeps-abs}
\| e_{k_\eps} \| \ \le \ \| f(P) \psi_{k_\eps} \| \, + \, \eps\, k_\eps\, .
\end{equation}
Next we estimate the term $P f(P) \psi_{k_\eps}$. From \eqref{gl:ekeps-ident} 
follows
\begin{eqnarray}
\| P f(P) \psi_{k_\eps} \|
 &  =  & \big\| P e_{k_\eps} - \big[ \big( I - (I-P)^{k_\eps-1} \big)\,
         (z_{f,g} - z_\eps) \big] \big\| \nonumber \\
 & \le & \| P e_{k_\eps} - (z_{f,g} - z_\eps) \| + \eps \nonumber \\
 &  =  & \| (I - T_l) \vphi_{k_\eps}^\eps - z_\eps \| + \eps \nonumber \\
 & \le & (\mu + 1) \eps . \label{gl:pfp-abs-le}
\end{eqnarray}
Further, we have from Lemma~\ref{lem:app3}
\begin{eqnarray}
\| P f(P) \psi_{k_\eps} \|^2 & & \nonumber \\
&\mbox{} \hspace{-4.2cm} =  & \hspace{-2.0cm}
 \exp(1)^{-2} \int_0^1 \exp \big( -[ (1-\ln(\lbd))^{-2p} ]^{-1/(2p)} \big)^2\,
 (1-\ln(\lbd))^{-2p}\, d\|E_\lbd \psi_{k_\eps}\|^2 \nonumber \\
&\mbox{} \hspace{-4.2cm} \ge  & \hspace{-2.0cm}
 \exp(1)^{-2}\, \hat{h} \Big( \int_0^1 (1-\ln(\lbd))^{-2p}\,
 d\|E_\lbd \psi_{k_\eps}\|^2 \Big) \nonumber \\[1ex]
&\mbox{} \hspace{-4.2cm} =  & \hspace{-2.0cm} \label{gl:pfp-abs-ge}
 \exp(1)^{-2}\, \hat{h} \big( \| f(P) \psi_{k_\eps} \|^2 \big) .
\end{eqnarray}
Then we obtain from \eqref{gl:pfp-abs-le}, \eqref{gl:pfp-abs-ge} and 
Lemma~\ref{lem:app4}
\begin{equation} \label{gl:fpkeps-abs}
\|f(P) \psi_{k_\eps}\| \ \le \ O \big( (-\ln (\eps^\frac{2}{3}))^{-p} \big) .
\end{equation}
Now we estimate the term $\eps k_\eps$. From assertion {\it i)} and 
Lemma~\ref{lem:app5} follows
\begin{equation} \label{gl:keps-abs}
k_\eps \ = \ O \big( \eps^{-1}\, (-\ln \sqrt{\eps})^{-p} \big) .
\end{equation}
Now, from \eqref{gl:ekeps-abs}, \eqref{gl:fpkeps-abs}, \eqref{gl:keps-abs} 
follows
$$ \| e_{k_\eps} \| \ \le \ O \big( (-\ln (\eps^\frac{2}{3}))^{-p} \big)
   \, + \, \eps\, O \big( \eps^{-1}\, (-\ln \sqrt{\eps})^{-p} \big)
   \ \le \ O \big( (-\ln \sqrt{\eps})^{-p} \big) , $$
proving the second assertion.
\end{proof}

\begin{remark} \label{rem:sc-interpr}
Next we consider the question of how to interpret the source condition 
defined in \eqref{gl:source-cond}. Let $\Omega$ be the square $[-\pi,\pi] 
\times [-\pi,\pi]$, $\Gamma_1 = \{ (x,y) \in \partial \Omega;\ x = -\pi \}$, 
$\Gamma_2 = \{ (x,y) \in \partial \Omega;\ x = \pi \}$ and consider the 
Cauchy problem
$$ \left\{ \begin{array}{ll}
     \Delta u = 0     &,\ \mbox{in} \ \Omega \\
     u = f            &,\ \mbox{at} \ \Gamma_1 \\
     u_\nu = g        &,\ \mbox{at} \ \Gamma_1 \\
     u(x,\pm\pi) = 0  &,\ x \in (-\pi,\pi)
   \end{array} \right. $$
where $f(y)=\sum_{j=1}^N a_j\sin(jy)$,\, $g(y) = \sum_{j=1}^N b_j\sin(jy)$. 
Given the Neumann data $\vphi(y) = \sum_{j=1}^{N}{\vphi_j \sin(jy)}$, we can 
explicitly represent the operator $T$ defined in \eqref{T_def} by
\begin{equation} \label{T_explizit}
   (T \varphi)(y)\ =\ \sum_{j=1}^{N}{\Big[ {\Big( \frac{\alpha_j}{\beta_j}
                                              \Big)}^2 \varphi_j
                   + \Big( \frac{j\,\alpha_j}{\beta_j^2}\Big) a_j
                   + \frac{1}{\beta_j}\ b_j \Big] \sin(jy)}\, ,
\end{equation}
where $\alpha_j = sinh(2j\pi)$ and $\beta_j = cosh(2j\pi)$. Now we define 
the Sobolev spaces of periodic functions
$$ H^s_{per}(-\pi,\pi) \ := \ \{ \vphi(y) = \T\sum\limits_{j\in\Z} 
   \vphi_j\ e^{ijy} \ | \ \sum\limits_{j\in\Z} (1+j^2)^s \vphi_j^2
   \, < \, \infty \}\, ,\ s\in\R . $$
If the Cauchy data $(f,g)$ is sufficiently regular, the Maz'ya iteration is 
well defined at $H^{-\me}_{per}(\Gamma_1)$ and we obtain from 
\eqref{T_explizit} a spectral representation of the linear part of the 
operator $T$:
\begin{equation} \label{Tl_explizit}
(T_l \, \vphi)(y) \ = \ \sum_{j=1}^\infty \lbd_j \, \vphi_{0,j} \, \sin(jy)
\end{equation}
with $\lbd_j = (\alpha_j / \beta_j)^2$, 
at the space $H := \overline{\mbox{Span}\{ \sin(jy),\ j\in\N \}}^ 
{ ||\cdot||_{H^{-\me}_{per}} }$. From the estimate
\begin{eqnarray*}
\ln \Big( \frac{\exp(1)}{1 - \lbd_j} \Big)
 & \ge & 1 - \ln \Big(\exp(1) \Big[ 1 - \frac{\alpha_j}{\beta_j} \Big] \Big) \\
 &  =  & - \ln \Big( \frac{2\exp(-2j\pi)}{\exp(2j\pi)+\exp(-2j\pi)} \Big)
 \ \ge \ 4\pi j - 1
\end{eqnarray*}
and from the source condition \eqref{gl:source-cond} follows
\begin{eqnarray*}
\| \bar{\vphi} - \vphi_1 \|_p^2
 &  =  & \T\sum\limits_{j=1}^\infty \D(1+j^2)^p \,
         \ln \Big( \frac{\exp(1)}{1 - \lbd_j} \Big)^{-2p} \psi_j^2 \\
 & \le & \T\sum\limits_{j=1}^\infty (1+j^2)^p \, (4\pi j - 1)^{-2p} \,
         \psi_j^2 \ \le \ c \sum\limits_{j=1}^\infty \psi_j^2 \ < \ \infty .
\end{eqnarray*}
This means that the source condition \eqref{gl:source-cond} can be interpreted 
as a regularity condition in the sense of $H^p$ spaces, i.e. $\bar{\vphi} - 
\vphi_1 \in H^p_{per}(\Gamma_1)$. This shows that a logarithmic source 
condition is indeed appropriate for our problem.
\end{remark}
%
%
%
%-----------------------------------------------------------------------------
\section{Numerical experiments} \label{sec:munerics}

Next we present some numerical results related to the numerical 
implementation of the Mann--Maz'ya iteration. The first two problems concern 
consistent Cauchy problems in a square and in an annular domain. In the third 
example we consider a problem with noisy data.

The computation was performed on the Silicon Graphics SGI-machines (based 
on R12000 processors; 32-bit code) at the Spezialforschungsbereich F013. 
The elliptic mixed boundary value problems were solved using the (NETLIB) 
software package PLTMG.%
\footnote{See URL\, {\em http://www.netlib.org}\, for details.}

\subsection{A consistent problem in a rectangular domain} \label{ssec:num1}

Let $\Omega \subset \R^2$ be the open rectangle $(0,1) \times (0,3/4)$ and 
define the following subsets of $\partial\Omega$:
$$         \Gamma_1 := \{ (x,0) ; x \in (0,1) \}\, , \ \ \ \ \ \ 
           \Gamma_2 := \{ (x, 3/4) ; x \in (0,1) \}\, , $$
$$        \Gamma_3 := \{ (0,y) ; y \in (0,3/4) \}\, , \ \ \ \,
          \Gamma_4 := \{ (1,y) ; y \in (0,3/4) \} . $$
We consider the Cauchy problem
$$  \left\{ \begin{array}{ccl}
            \Delta \, u  & = & 0 \ ,\ \mbox{ in } \Omega \\
            u            & = & f \ ,\ \mbox{ at } \Gamma_1 \\
            u_{\nu}      & = & g \ ,\ \mbox{ at } \Gamma_1 \\
            u            & = & 0 \ ,\ \mbox{ at } \Gamma_3 \cup \Gamma_4
            \end{array} \right.  $$
where the Cauchy data\, $f(x)= \sin(\pi x)$,\, $g \equiv 0$\, is given at 
$\Gamma_1$. We aim to reconstruct the (Dirichlet) trace of $u$ at $\Gamma_2$. 
As one can easily check, the exact solution of this Cauchy Problem is given 
by\, $\bar{u}(x,y) = \cosh(\pi y) \, \sin(\pi x)$.

We used in the iteration the matrix $A = (a_{jk})$, with $a_{jk} = k^{-1}$ 
for $j \le k$. As initial guess, we chose $\vphi_1 \equiv 0$. Each mixed 
boundary value problem was solved using (multi-grid) finite element methods, 
with linear elements and a uniform mesh with 65\,921 nodes (256 nodes on 
$\Gamma_2$). We used the stopping rule $\| \vphi_k - \vphi_{k-1} \|_ 
{L^2(\Gamma_2)} \le 10^{-3}$. 
In Figure~\ref{fig:rectang1} we present the results corresponding 
to the Mann--Maz'ya iteration: the dotted line represents the exact solution 
and the solid line represents the sequence generated by $(0,A,T)$.

\begin{figure}[t] \unitlength1cm
\begin{center}
\begin{picture}(14,4)
%\put(0,0){\dashbox{0.1}(14,4){}}
\centerline{
\epsfxsize3.3cm\epsfysize4cm \epsfbox{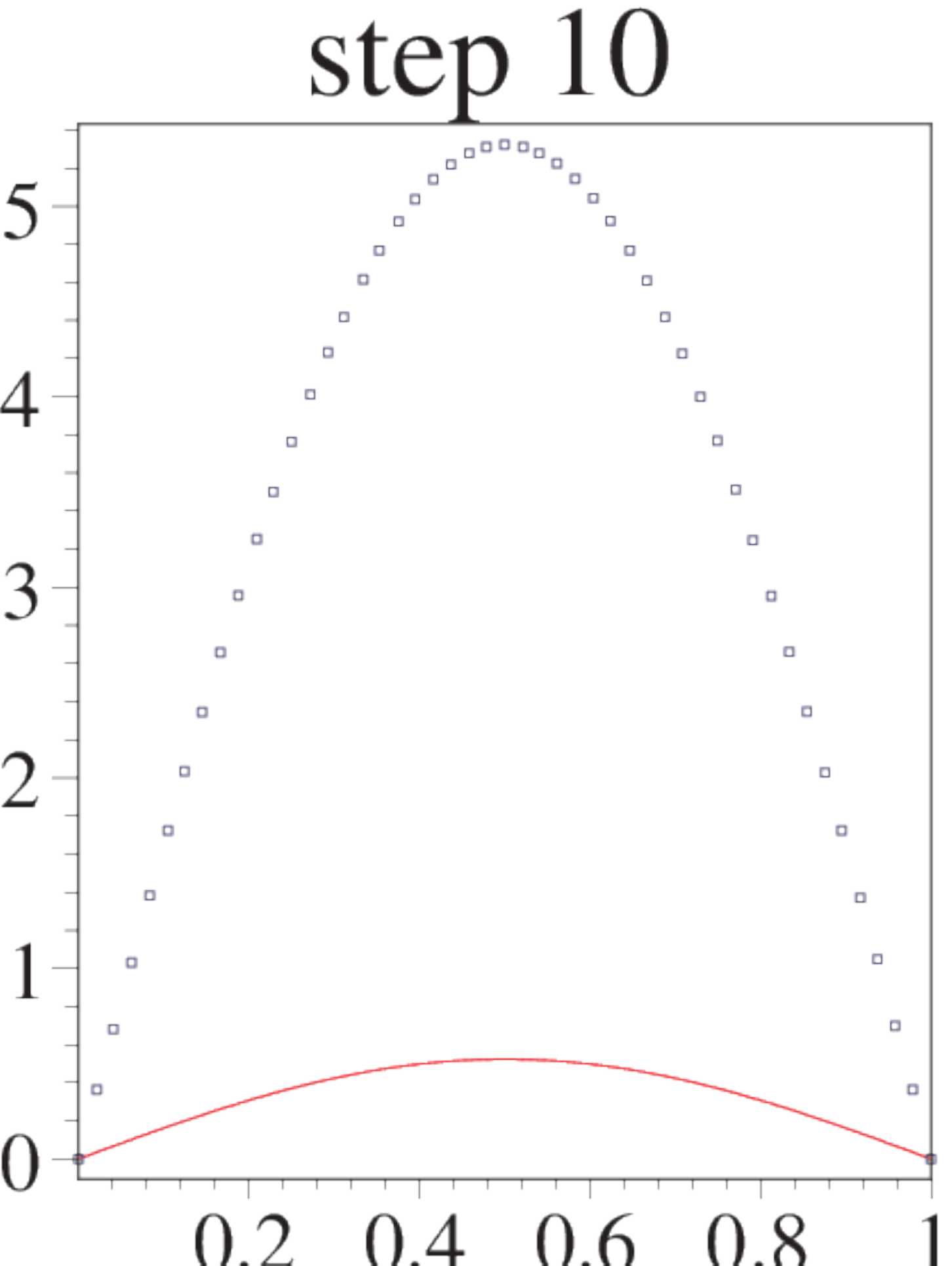}
\epsfxsize3.3cm\epsfysize4cm \epsfbox{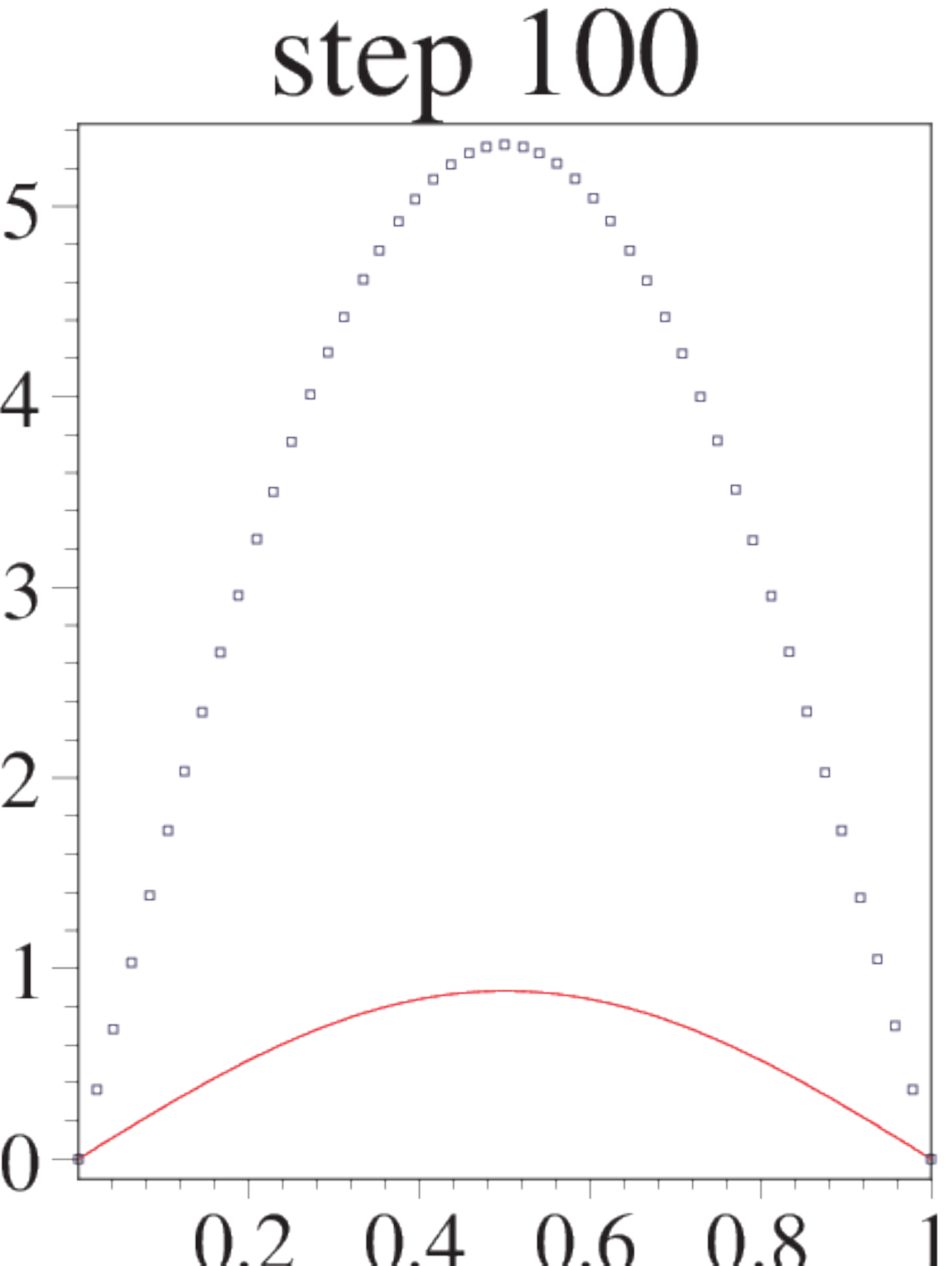}
\epsfxsize3.3cm\epsfysize4cm \epsfbox{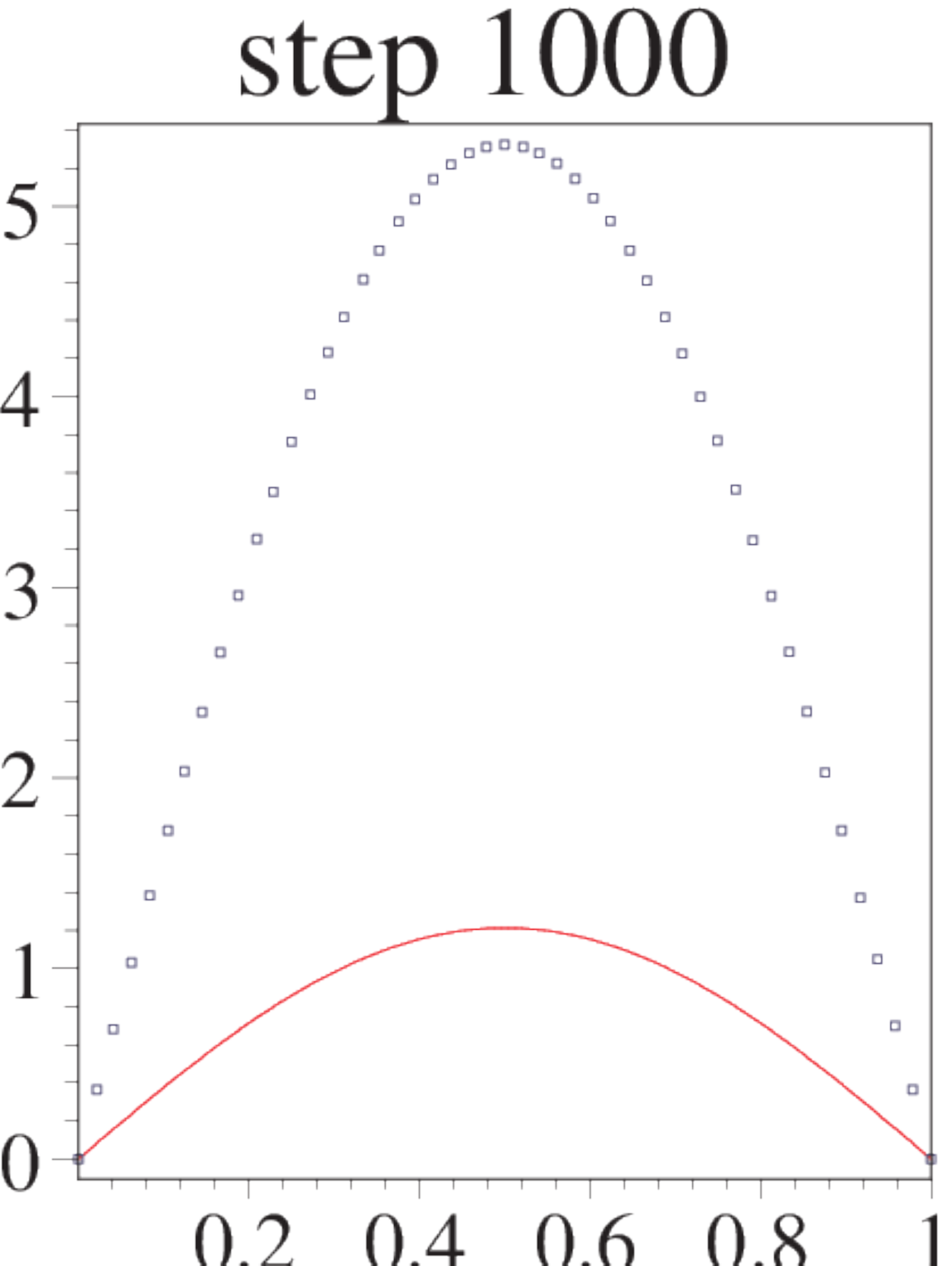}
\epsfxsize3.3cm\epsfysize4cm \epsfbox{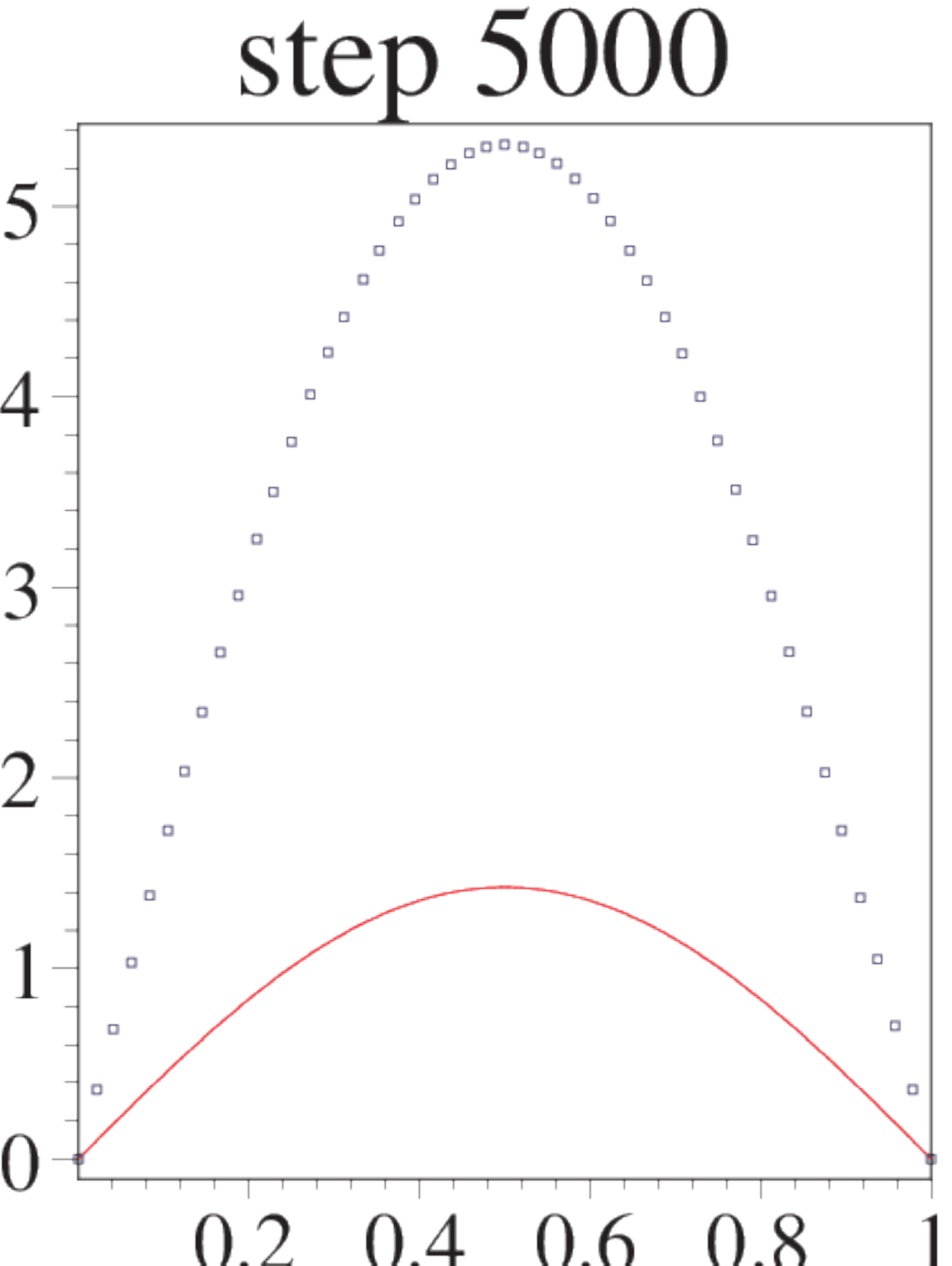} }
\end{picture}
\end{center} \vskip-0.6cm
\caption{Rectangular domain; iteration for consistent Cauchy data 
\label{fig:rectang1}}
\end{figure}

As one can observe in Figure~\ref{fig:rectang1}, the convergence rate decays 
very fast. This can be in part explained by linear convex combination used 
to compute $\psi_k$ in the Mann--Maz'ya iteration (note that\, $\psi_{k+1} 
- \psi_k = \frac{1}{k+1} \vphi_{k+1} - \frac{1}{k(k+1)} \sum_{j=1}^k\vphi_k$). 
As an alternative to improve the convergence rate, we relaxed the stopping 
rule and restarted the iteration, using the last evaluated $\vphi_k$ 
as new initial guess. In Figure~\ref{fig:rectang2} we present the results 
corresponding to this restart strategy. For comparison purposes, we 
restarted the iteration after every 50 steps. Thus, to compute the results 
in Figure~\ref{fig:rectang2}, we had to evaluate 50, 100, 250 and 500 
iteration steps respectively.

\begin{figure}[b] \unitlength1cm
\begin{center}
\begin{picture}(14,4)
%\put(0,0){\dashbox{0.1}(14,4){}}
\centerline{
\epsfxsize3.3cm\epsfysize4cm \epsfbox{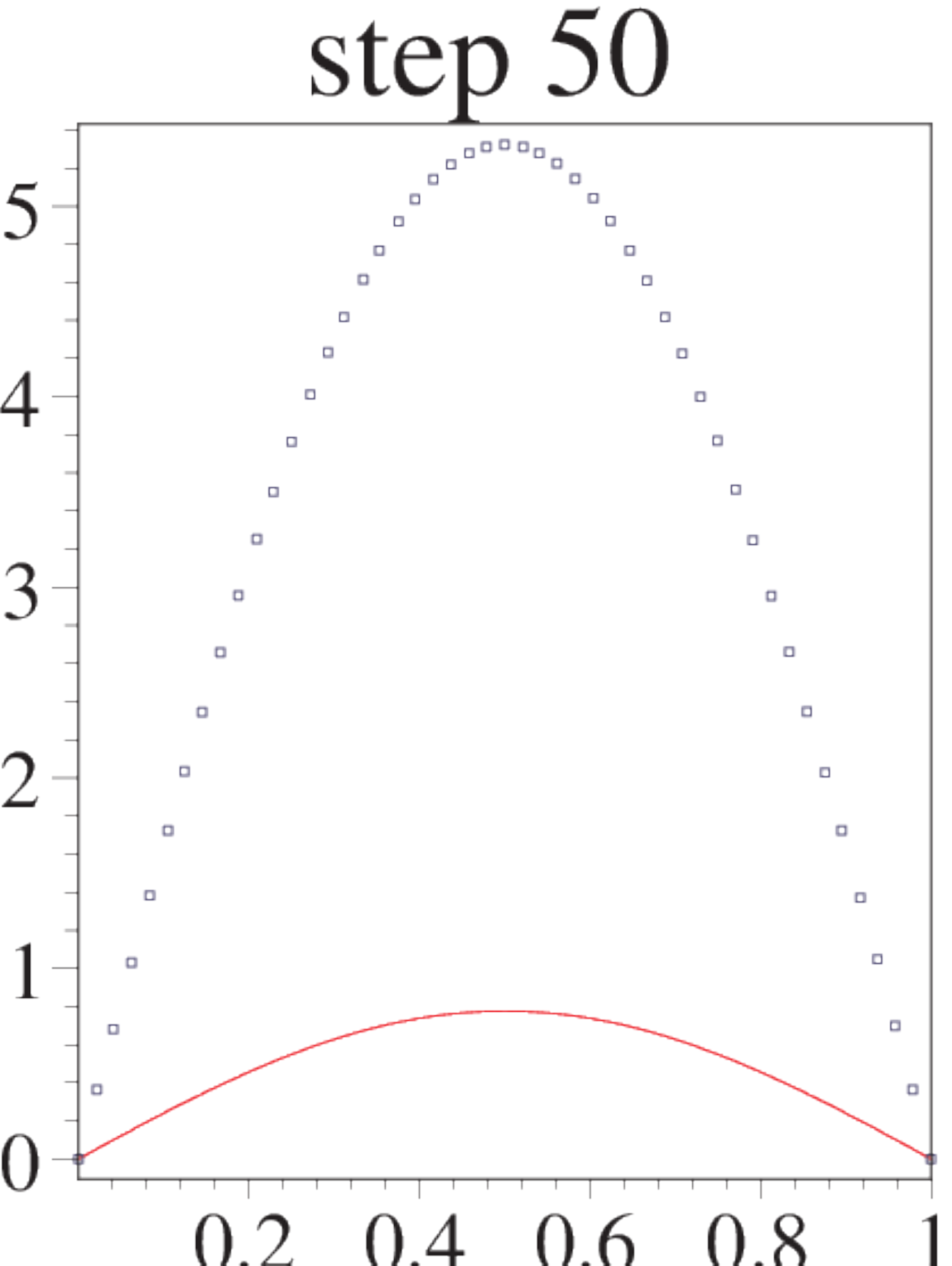}
\epsfxsize3.3cm\epsfysize4cm \epsfbox{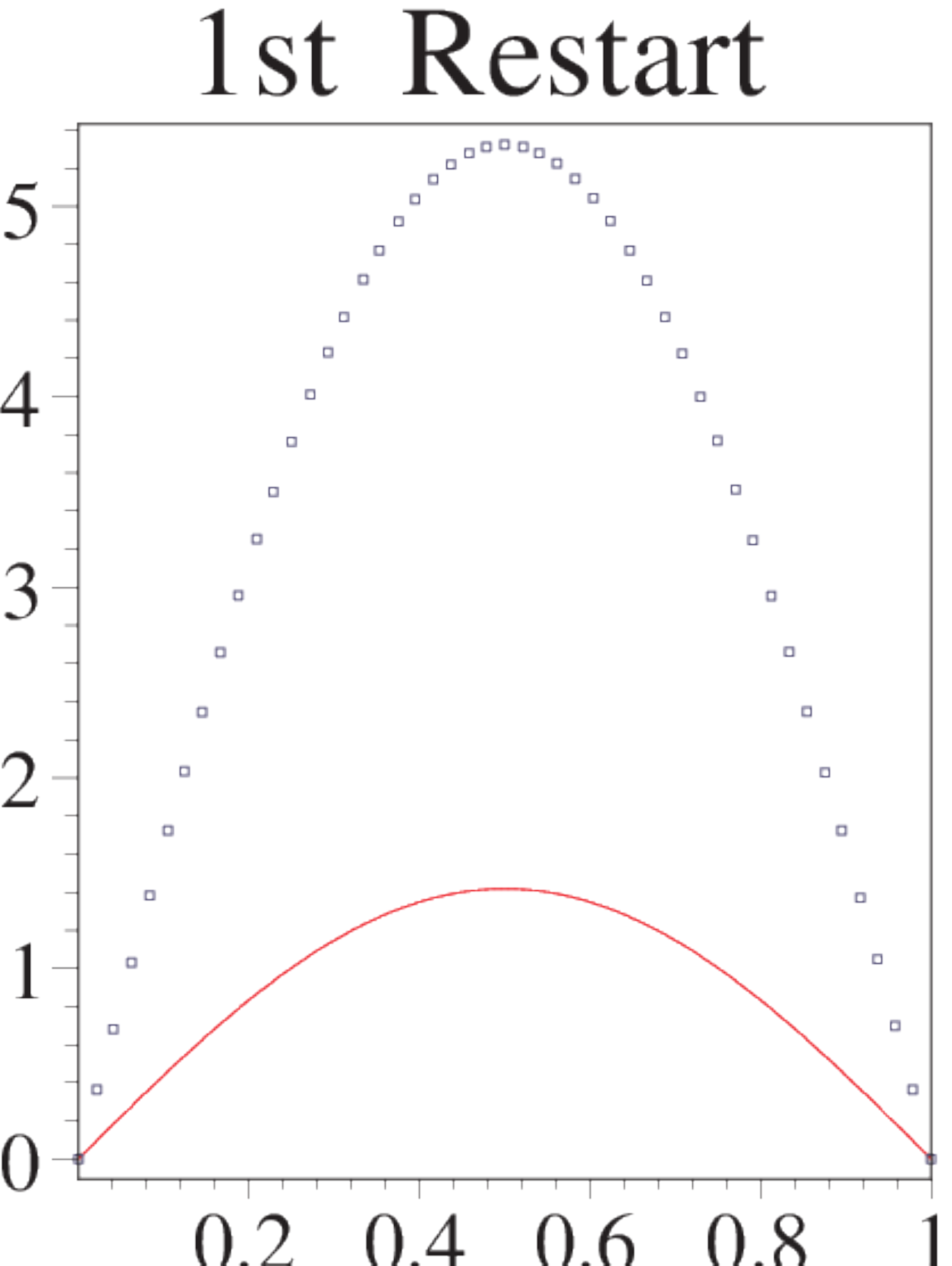}
\epsfxsize3.3cm\epsfysize4cm \epsfbox{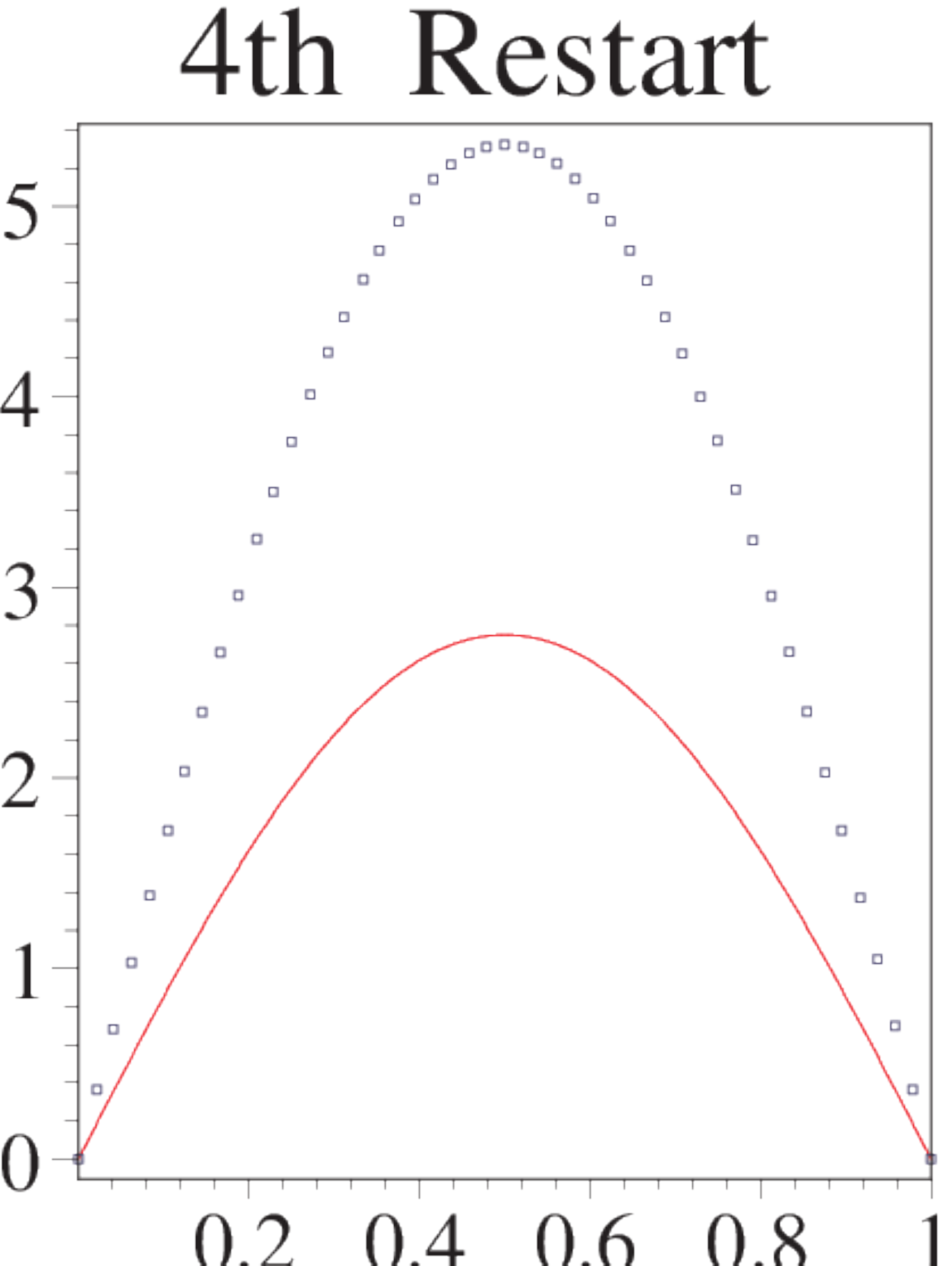}
\epsfxsize3.3cm\epsfysize4cm \epsfbox{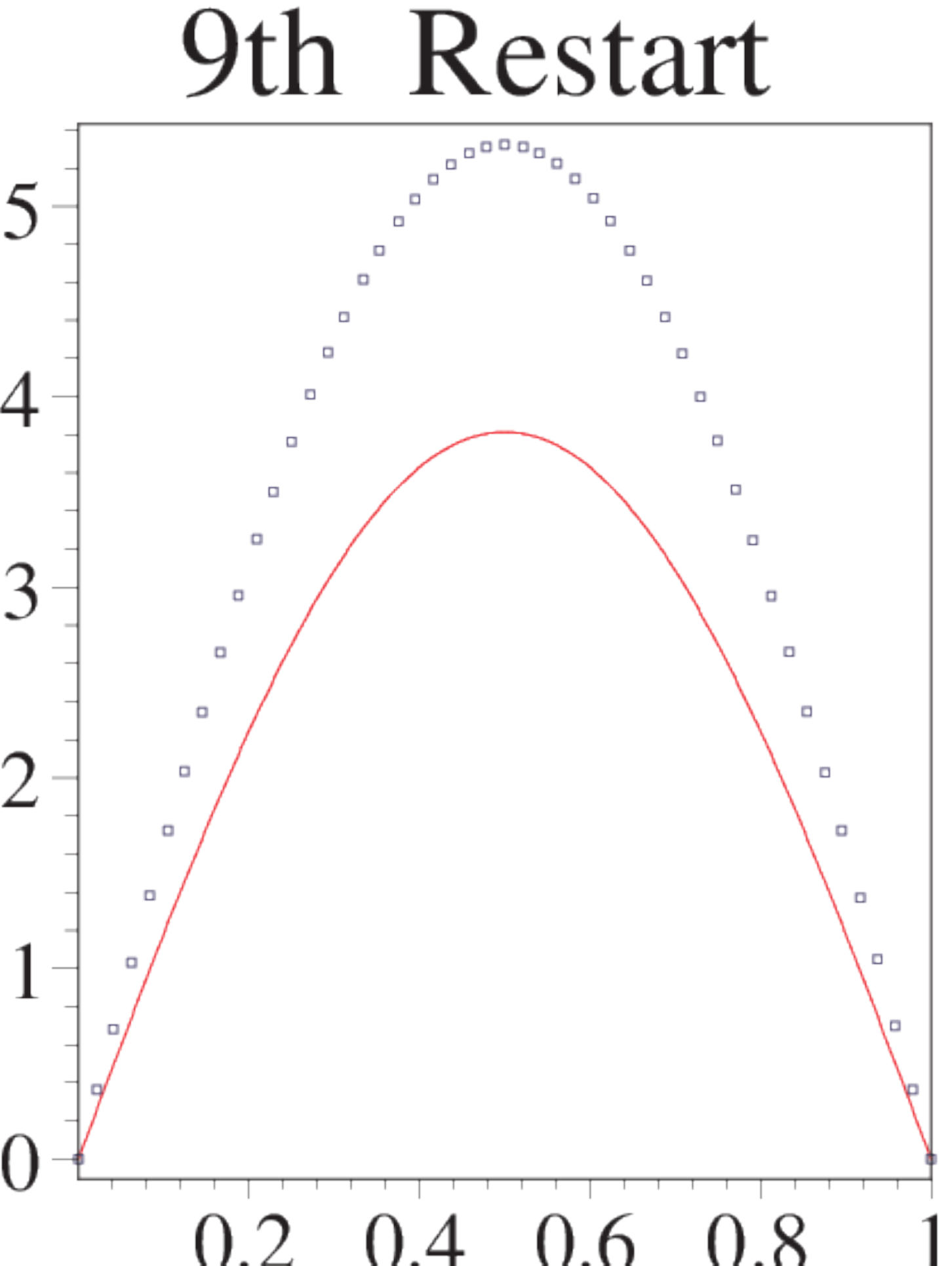} }
\end{picture}
\end{center} \vskip-0.6cm
\caption{Rectangular domain; iteration with restart strategy 
\label{fig:rectang2}}
\end{figure}

The restart strategy seams to save a considerable amount of computation 
effort, however we should quote that we have no analytical justification 
neither for the choice of the restart criterion nor for the improvement in the 
convergence rate. (Argumenting as in Section~\ref{ssec:regul}, one can 
verify that the sequence of residuals is again non-increasing.)

\subsection{A consistent problem in an annular domain} \label{ssec:num2}

Let $\Omega$ be the annulus centered at the origin with inner and outer 
radius respectively 1 and 3. We denote the inner and outer boundaries by 
$\Gamma_1$ and $\Gamma_2$ respectively. We consider the following Cauchy 
problem:
$$  \left\{ \begin{array}{ccl}
            \Delta \, u  & = & 0 \ ,\ \mbox{ in } \Omega \\
            u            & = & f \ ,\ \mbox{ at } \Gamma_1 \\
            u_{\nu}      & = & g \ ,\ \mbox{ at } \Gamma_1 \\
            \end{array} \right.  $$
Given the Cauchy data\, $f(\theta) = \sin(\theta) - \sin(2\theta)/2$\, and\, 
$g \equiv 0$\, at $\Gamma_1$, we want to reconstruct at $\Gamma_2$ the trace 
of $u$. It is easy check, that the solution of this Cauchy problem is given 
by\, $\bar{u}(r,\theta) = \frac{1}{2}(r+r^{-1}) \sin(\theta) - \frac{1}{4} 
(r^2+r^{-2}) \sin(2\theta)$.

The matrix $A$ is chosen as in Section~\ref{ssec:num1}. For the multi-grid 
method we used linear elements and a (uniform) mesh with 61\,824 nodes (512 
nodes on $\Gamma_2$). We use the same initial guess $\vphi_1 \equiv 0$ and the 
same stopping rule as in the previous example. In Figure~\ref{fig:annulus} we 
present the results corresponding to the Mann--Maz'ya iteration: the dotted 
line represents the exact solution and the solid line represents the sequence 
generated by the iteration $(0,A,T)$. (Note that in Figures~\ref{fig:annulus}, 
\ref{fig:noisy-data} and \ref{fig:noise1} the horizontal axis is parameterized 
from 0 to $2\pi$.)

\begin{figure}[h] \unitlength1cm
\begin{center}
\begin{picture}(14,4)
%\put(0,0){\dashbox{0.1}(14,4){}}
\centerline{
\epsfxsize3.3cm\epsfysize4cm \epsfbox{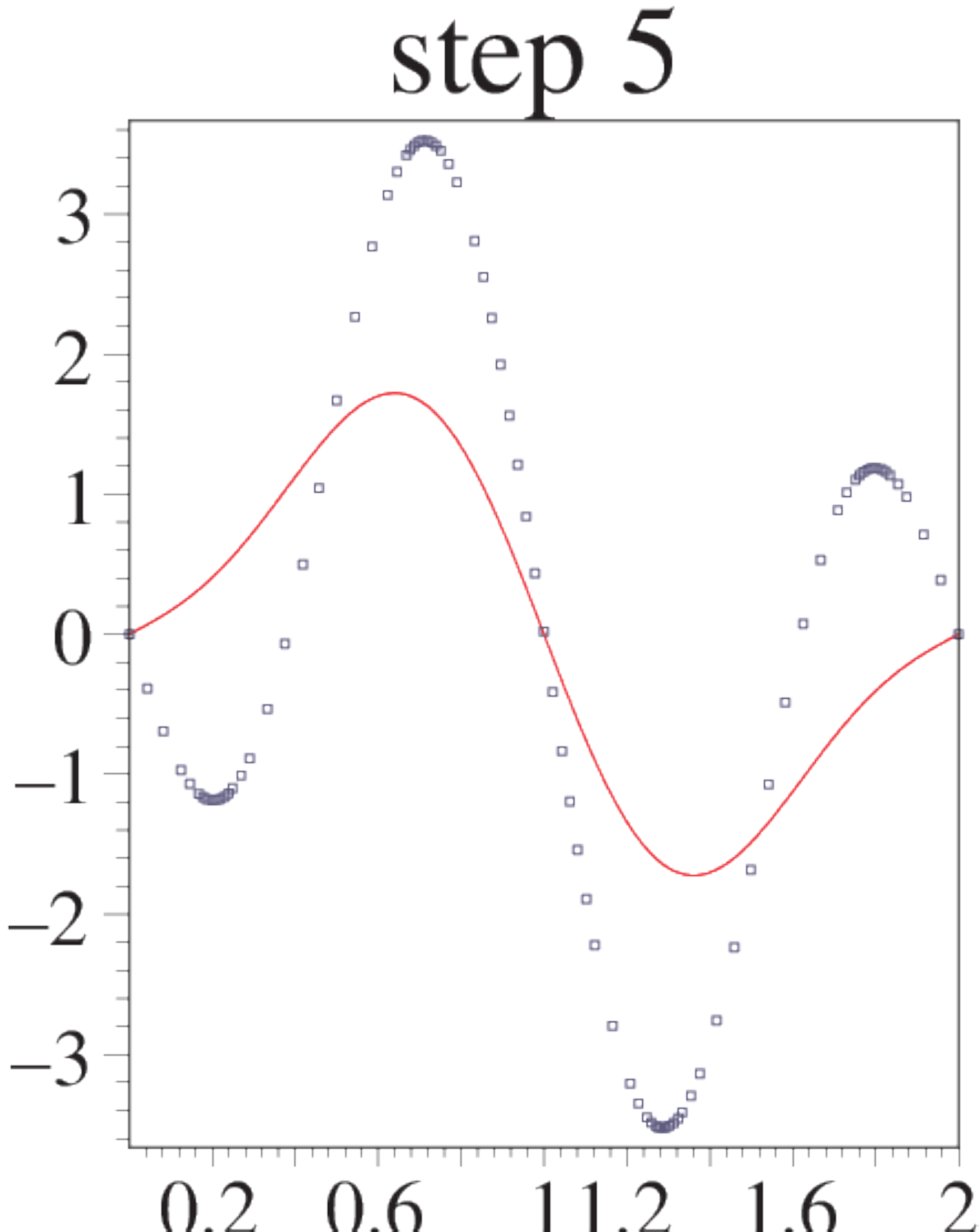}
\epsfxsize3.3cm\epsfysize4cm \epsfbox{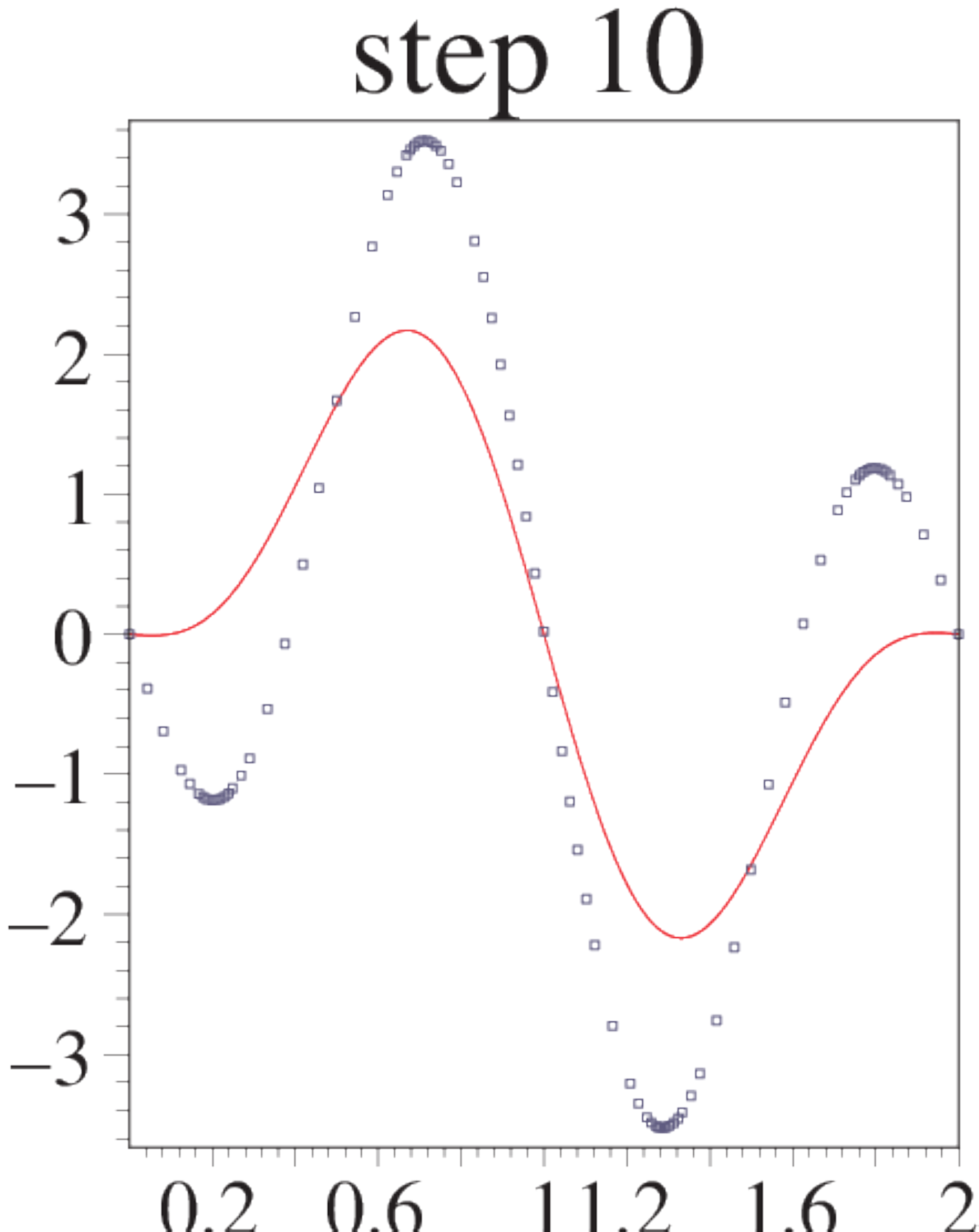}
\epsfxsize3.3cm\epsfysize4cm \epsfbox{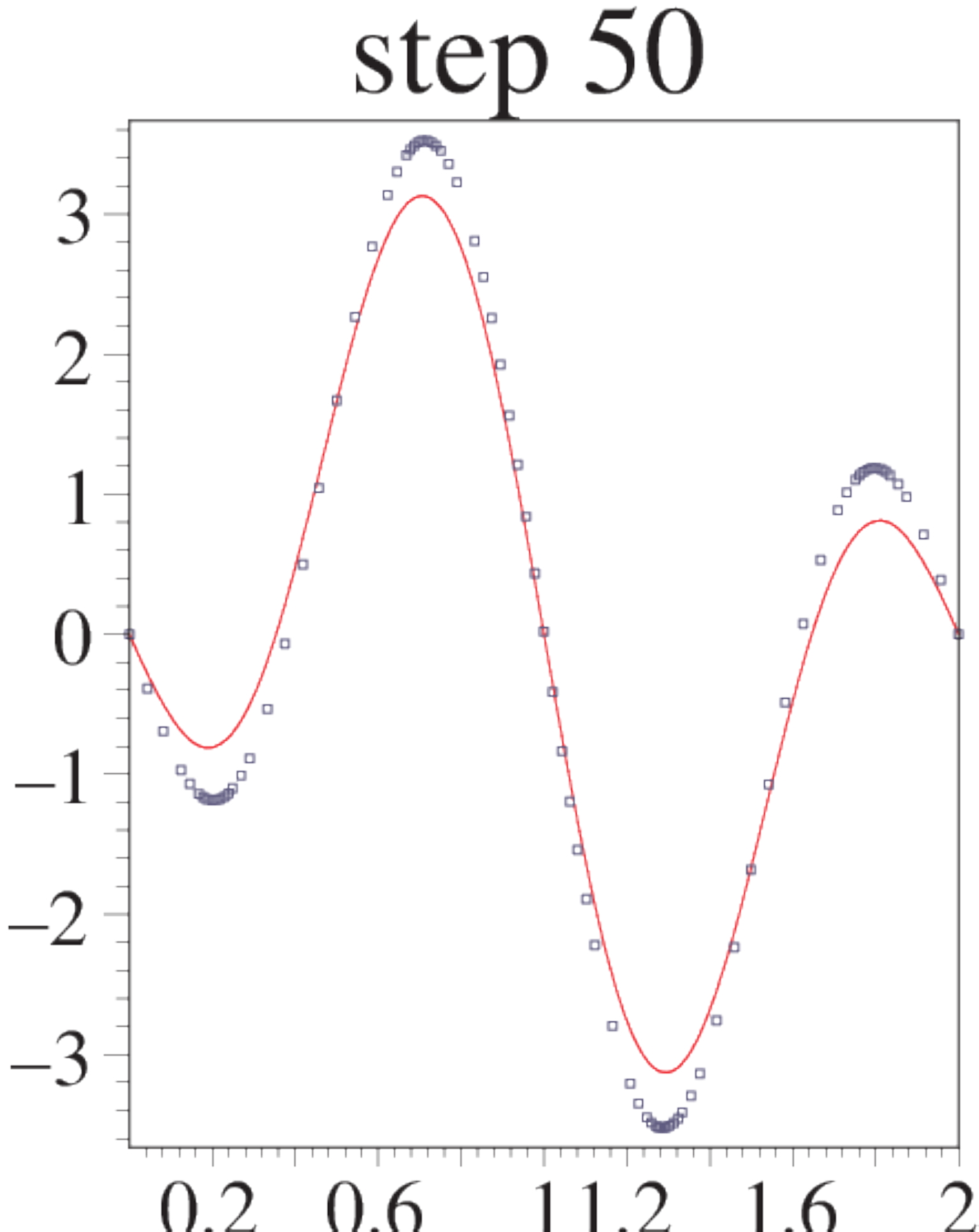}
\epsfxsize3.3cm\epsfysize4cm \epsfbox{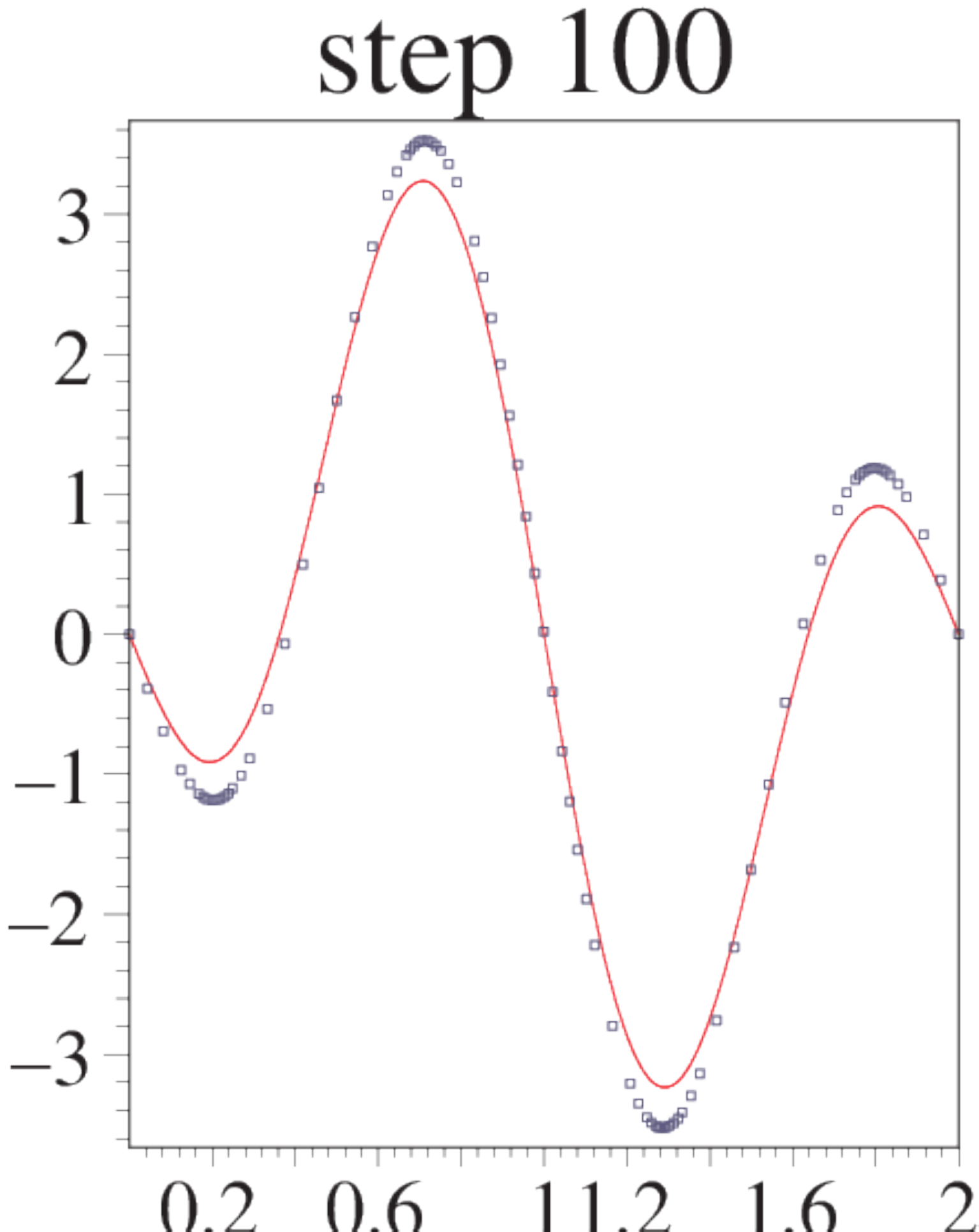} }
\end{picture}
\end{center} \vskip-0.6cm
\caption{Annular domain; iteration for consistent Cauchy data 
\label{fig:annulus}}
\end{figure}

\subsection{An inconsistent problem (noisy data)} \label{ssec:num3}

For this experiment we consider once more the Cauchy problem formulated 
in Section~\ref{ssec:num2}. The noisy data is obtained by inserting in the 
exact Cauchy data $(f,g) = (\sin(\theta)-\frac{1}{2}\sin(2\theta),\, 0)$ a 
perturbation of 5\%. In Figure~\ref{fig:noisy-data} we present the 
perturbations added to the Dirichlet and to the Neumann data.

The matrix $A$ is chosen as in Section~\ref{ssec:num1}. The initial guess, 
stopping rule and mesh level used for the computation are the same as those 
considered in that section. The the results corresponding to the 
Mann--Maz'ya iteration are presented in Figure~\ref{fig:noise1} (thick line).

%###################################
\unitlength1mm
\begin{wrapfigure}{r}{8.3cm}
\begin{center} {\vspace{-1.2ex}
\begin{center}
\mbox{\epsfxsize4cm\epsfysize3cm \epsfbox{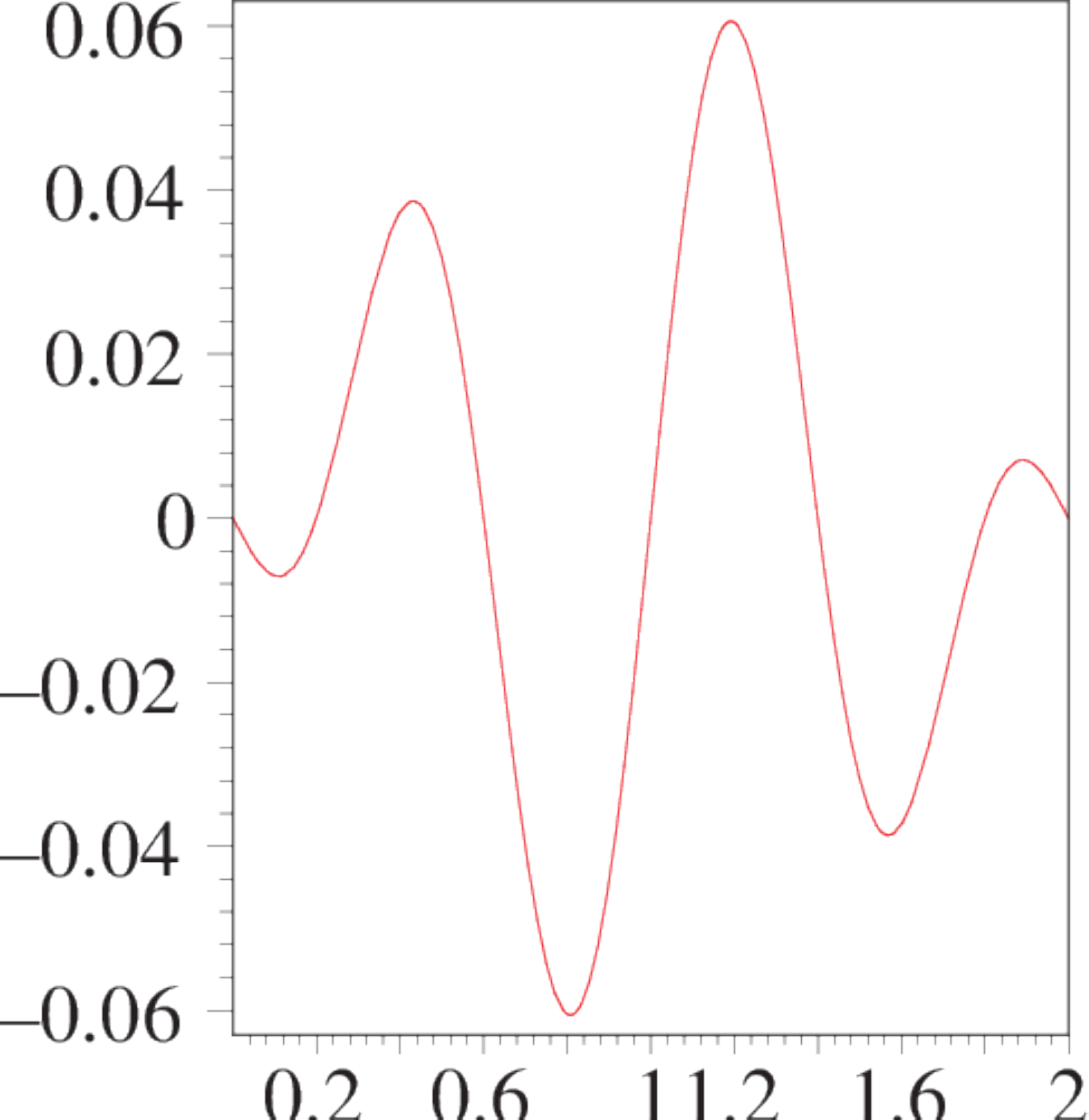}
      \epsfxsize4cm\epsfysize3cm \epsfbox{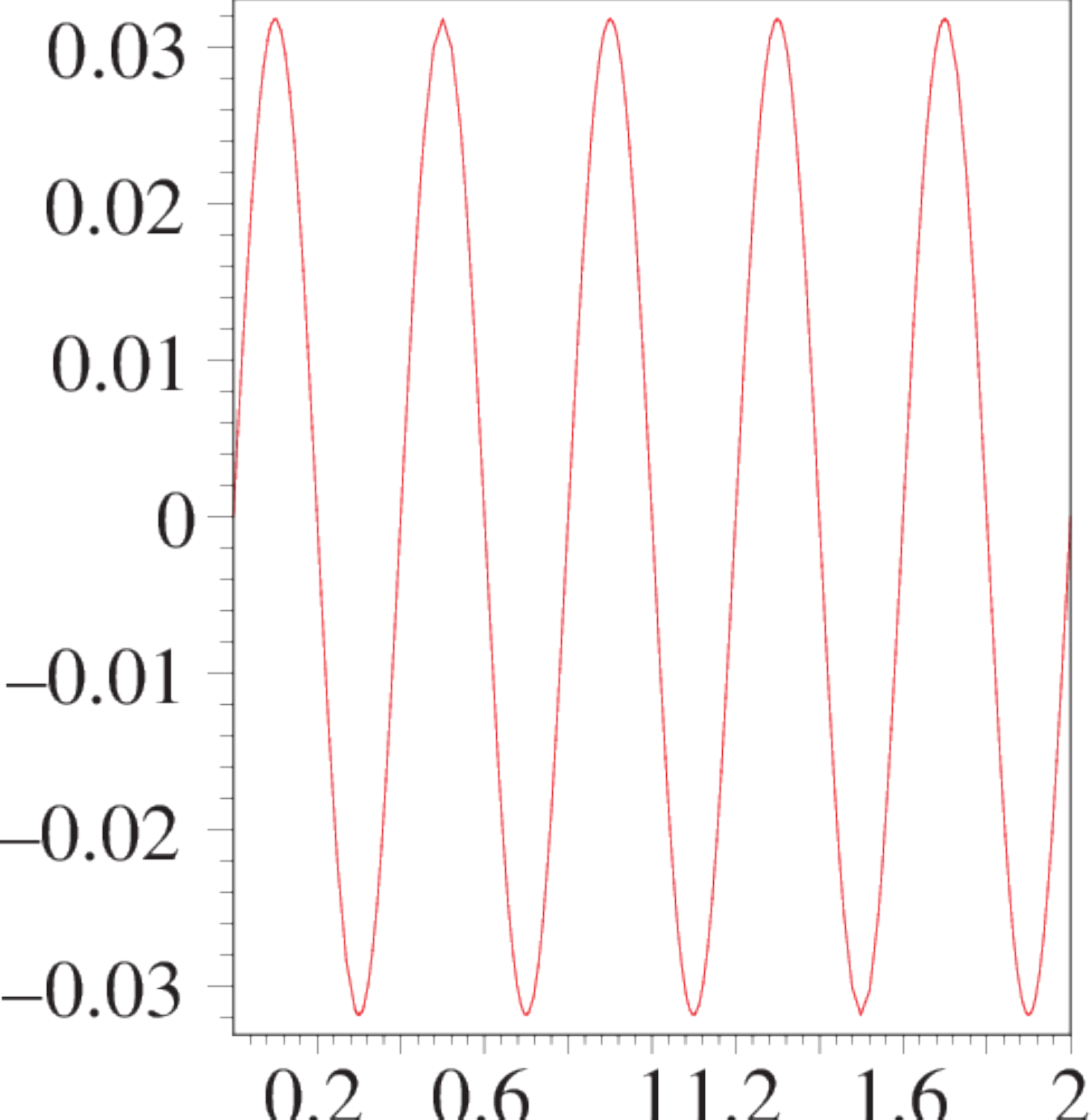} }
\end{center}
\centerline{\hspace{2cm} (a) \hspace{3.4cm} (b) \hfill \mbox{}}
\vspace{-0.5ex}
\caption{Generation of noisy data; (a) Perturbation added to the Dirichlet 
data; (b) Perturbation added to the Neumann data \label{fig:noisy-data} }
\vspace{-4ex}}
\end{center}
\end{wrapfigure}
%###################################
As one can see in Figure~\ref{fig:noise1}, the performance of the numerical 
implementation of the Mann--Maz'ya iteration is stable. The high frequency 
components of the error start to interfere in the iteration only after an 
exponential number of steps. Indeed, in the Maz'ya algorithm the iteration 
error $e_k := \bar{\vphi} - \vphi_k^\eps$ satisfies $e_{k+1} = T_l \, e_k = 
T_l^k e_1$, and the eigenvalues of the fixed point operator $T_l$ converge 
exponentially to 1 (see Remark~\ref{rem:sc-interpr} for an example). Thus, 
the iteration reconstructs first the projection of the solution over the 
first eigenspace of the operator $T_l$; after an exponential number of steps, 
it reconstructs also the projection over the second eigenspace; and so on.

The high frequency components of the error are exponentially amplified and 
destroy completely the approximation if we iterate long enough. However, due 
to the characteristic explained above, one has to evaluate an exponential 
number of steps in order to observe the influence of the {\em bad} frequencies.

\begin{figure}[b] \unitlength1cm
\begin{center}
\begin{picture}(14,4)
%\put(0,0){\dashbox{0.1}(14,4){}}
\centerline{
\epsfxsize3.3cm\epsfysize4cm \epsfbox{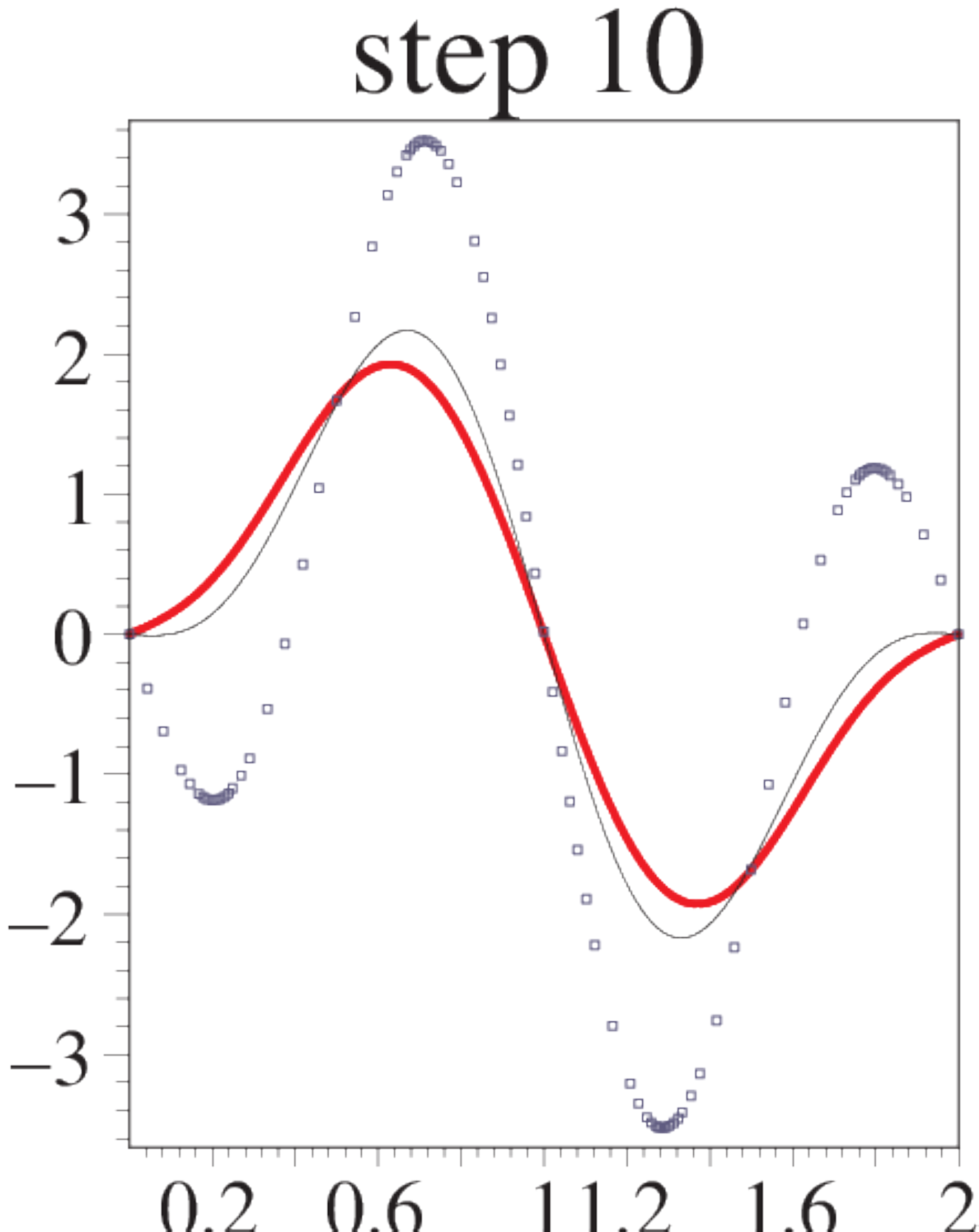}
\hskip1.6cm
\epsfxsize3.3cm\epsfysize4cm \epsfbox{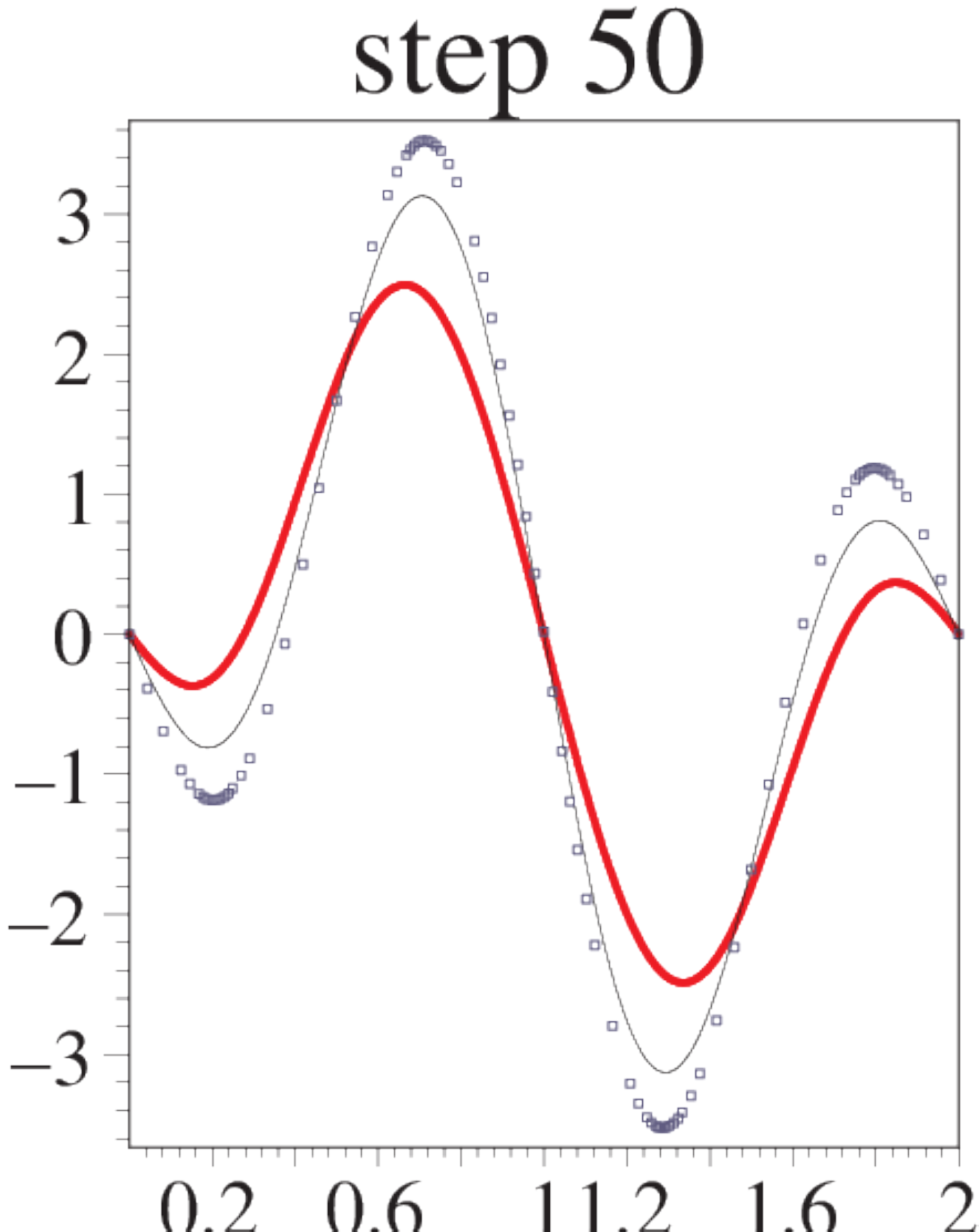}
\hskip1.6cm
\epsfxsize3.3cm\epsfysize4cm \epsfbox{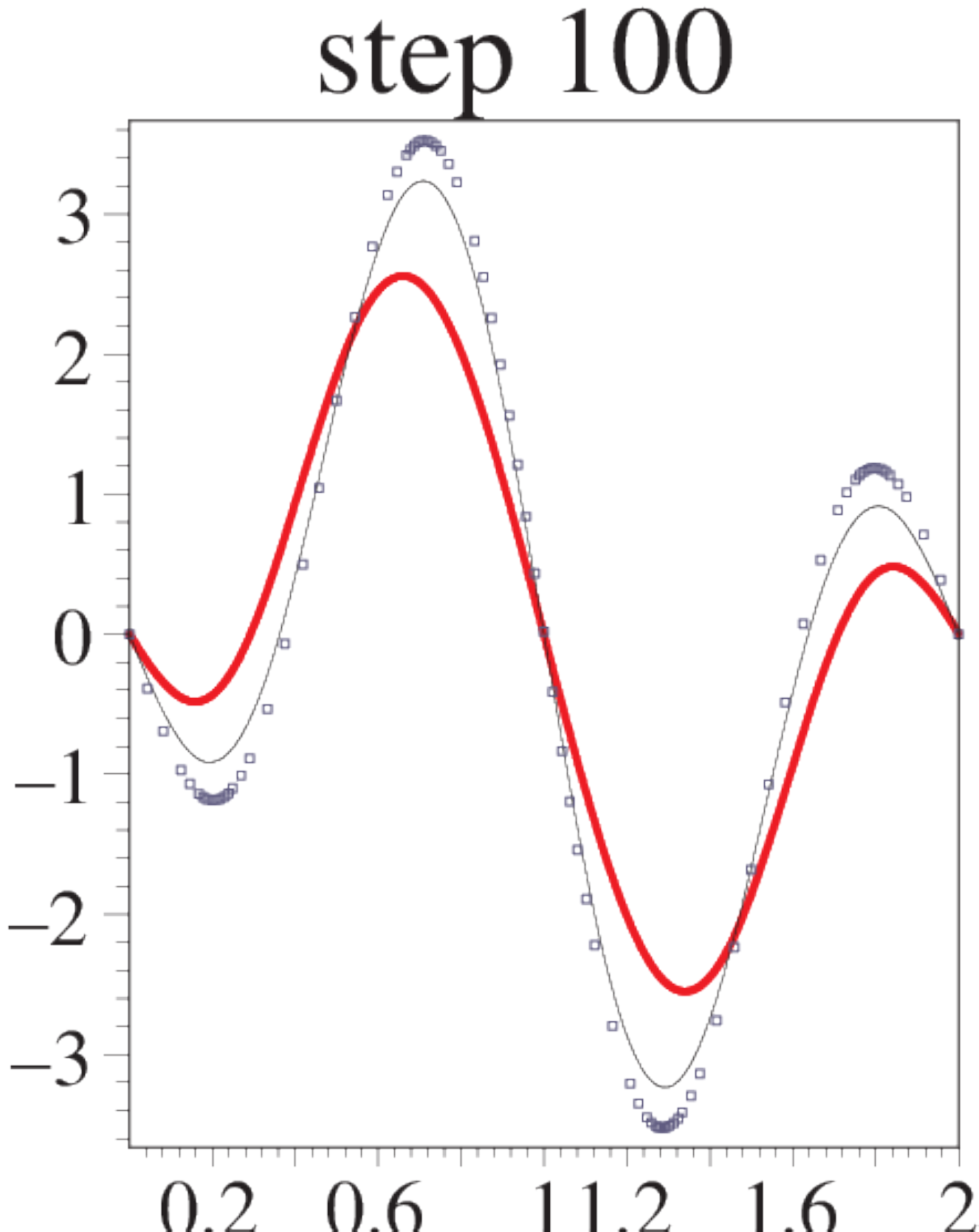} }
\end{picture}
\end{center} \vskip-0.6cm
\caption{Iteration for noisy data in the annular domain; dotted-line: exact 
solution; thin-line: iteration for exact data; thick-line: iteration for 
noisy data \label{fig:noise1} }
\end{figure}
%
%
%
%-----------------------------------------------------------------------------
\section{Final remarks} \label{sec:concl}

Let us suppose $\partial\Omega = \Gamma_1 \cup \Gamma_2 \cup \Gamma_3$ and 
we want to analyze a Cauchy problem with data given at $\Gamma_1$ plus some 
further boundary condition (Neumann, Dirichlet, $\dots$) at $\Gamma_3$ (this 
is precisely the situation in Section~\ref{ssec:num1}). It is still possible 
to use the Maz'ya method for such problems. All we have to do is to adapt 
the Maz'ya iteration by adding this extra boundary condition at $\Gamma_3$ 
to both mixed boundary value problems at each iteration step. This 
over-determination of boundary data does not affect the analysis of the 
Maz'ya algorithm (see \cite{Le}). Consequently, it also does not also affect 
the analysis of the Mann--Maz'ya iteration presented here.

The Maz'ya iterative method (see Section~\ref{ssec:mazya}) generates a 
sequence of Neumann traces, which approximate the unknown Neumann boundary 
condition $u_{\nuA|_{\Gamma_2}}$. Analogously, we can define an iterative 
method, which produces a sequence of Dirichlet traces (this was already 
suggested by Maz'ya et al. in \cite{KMF}). The convergence proof for this 
iteration is quite similar to the one presented in \cite{Le} for the Maz'ya 
iteration. It is also possible to combine this iteration with the Mann 
strategy. All the results formulated here for the Mann--Maz'ya method 
remain valid for this iteration.

The approach followed in \cite{Le} to characterize the solution of elliptic 
Cauchy problems as a solution of a fixed point equation can be extended, 
using spectral theory, to differential operators of hyperbolic and parabolic 
types (see \cite{BaLe}). The formulation of the Mann iteration for this 
problems follow the lines discussed here.
%
%
%
%-----------------------------------------------------------------------------
\appendix

\section*{Appendix}
\setcounter{section}{1}

\begin{applemma} \label{lem:app1}
Let $p > 0$ and $k \in \N_0$. Define the real-valued function 
$\hat{f}(\lbd) := (1 - \lbd)^k \big( \ln \frac{\exp(1)}{\lbd} \big)^{-p}$, 
$\lbd \in [0,1]$. Then we have
$$  \hat{f}(\lbd) \ \le \ C\, (\ln\, k)^{-p}\, ,\ \lbd \in [0,1]\, , $$
with $C$ independent of $k$. Moreover, for each $p \in \R$, the real-valued 
function $\hat{g}(\lbd) := (1 - \lbd)^k \lbd \big( \ln \frac{\exp(1)}{\lbd} 
\big)^{-p}$ defined on $[0,1]$, satisfies
$$  \hat{g}(\lbd) \ \le \ C\, k^{-1} (\ln\, k)^{-p}\, ,\ \lbd \in [0,1]\, , $$
with $C$ independent of $k$.
\end{applemma}
\begin{proof}
The first assertion is proved in \cite[Lemma~A.1]{DES}. The second assertion 
is quite similar to the second part of the lemma cited above and can be 
proved with an analogous argumentation.
\end{proof}

\begin{applemma} \label{lem:app2}
Let $p > 0$, $\cal H$ a Hilbert space, $P: {\cal H} \to {\cal H}$ a positive 
linear self adjoint non-expansive operator, $f$ the real-valued function 
defined by
$$ f(\lbd) \ := \ 
   \begin{cases}
     \Big( \ln (\exp(1) \lbd^{-1}) \Big)^{-p} &\!\!\!\!\!\!,\ 
     \lbd \in (0,1] \\
     \ \ \ \ \ \ \ \ \ \ \ 0 &\!\!\!\!\!\!,\ \lbd = 0 .
   \end{cases} $$
Let $e_1 := f(P) \psi$, for some $\psi \in {\cal H}$. Then for any $k > 1$
\begin{eqnarray*}
&&\| (I - P)^k e_1 \| \ \le \ C\, \| \psi \|\, \big(\ln(k+2)\big)^{-p} \\[1ex]
&&\| (I - P)^k P e_1 \| \ \le \ C\, \| \psi \|\, k^{-1} \big(\ln(k+2)\big)^{-p}
\end{eqnarray*}
with $C$ independent of $k$.
\end{applemma}

\begin{proof} 
This lemma is analogue to the result stated in \cite[Lemma~2.6]{DES}. It 
follows from Lemma~\ref{lem:app1}, the same way as \cite[Lemma~2.6]{DES} 
follows form \cite[Lemma~A.1]{DES}.
\end{proof}

\begin{applemma} \label{lem:app3}
Given $p > 1$, then
$$ \int_0^1 \hat{h} \big( (1-\ln(\lbd))^{-2p} \big)\, d\lbd \ \ge \ 
   \hat{h} \Big( \int_0^1 (1-\ln(\lbd))^{-2p}\, d\lbd \Big) , $$
with $\hat{h}(t) := \exp(-t^\frac{-1}{2p})^2\, t$.
\end{applemma}

\begin{proof}
The assertion follows from the convexity of $\hat{h}$ in $[0,\infty)$ and 
Jensen's inequality. (see \cite[Lemma~A.2]{DES})
\end{proof}

\begin{applemma} \label{lem:app4}
Let $p \ge 1$, $C > 0$ and $\eps >0$ sufficiently small such that $1 \ge 
(-\ln(C \eps^\frac{2}{3}))^{-2p} \ge \eps$. Let
$$ \int_0^1 \hat{h} \big( (1-\ln(\lbd))^{-2p} \big)\,
   d\|E_\lbd \psi\|^2 \ = \ C \eps^2 . $$
Then there exists $D > 0$ (independent of $\eps$) such that
$$ \int_0^1 (1-\ln(\lbd))^{-2p}\, d\|E_\lbd \psi\|^2 \ \le \ 
   D (-\ln (\eps^\frac{2}{3}))^{-2p} . $$
\end{applemma}

\begin{proof}
Let $s_2 = ( -\ln(C \eps^\frac{2}{3}) )^{-2p}$. Then
$$ \hat{h} (s_2) \ = \ \big(C \eps^\frac{2}{3})^2
   \big( -\ln(C \eps^\frac{2}{3}) \big)^{-2p}  \ \ge \ C \eps^2 . $$
Thus, it follows from the monotonicity of $\hat{h}$ that the equation 
$\hat{h}(s) = C \eps^2$ has a solution $s_1 \in (0,s_2]$. From the assumptions 
and Lemma~\ref{lem:app3} follows
$$ \hat{h} \Big( \int_0^1 (1-\ln(\lbd))^{-2p}\, d\|E_\lbd \psi\|^2
   \Big) \ \le \ C \eps^2 $$
and we conclude from the monotonicity of $\hat{h}$ that
$$ \int_0^1 \hat{h} \big( (1-\ln(\lbd))^{-2p} \big)\,
   d\|E_\lbd \psi\|^2 \ \le \ s_1 \ \le \ s_2 \ \le \ 
   D (-\ln (\eps^\frac{2}{3}))^{-2p} $$
for a generic constant $D$.
\end{proof}

\begin{applemma} \label{lem:app5}
Let $\hat{k}$ be a solution of
$$ k\, (\ln\, k)^{p} \ = \ C\, \eps^{-1}\, . $$
Then $\hat{k}$ satisfies
$$ \hat{k} \ = \ O \big( \eps^{-1}\, (-\ln\sqrt{\eps})^{-p} \big) . $$
\end{applemma}

\begin{proof}
Follows immediately from \cite[Lemma~A.6]{DES}
\end{proof}
%
%
%
%-----------------------------------------------------------------------------

\end{document}